\documentclass[10pt,leqno]{amsart}
\usepackage{amsmath,amsthm,amsfonts,eucal,eufrak}
\usepackage{latexsym}
\usepackage{amssymb}
\usepackage[dvips]{epsfig}
\usepackage[active]{srcltx}
\renewcommand{\thesection}{\arabic{section}}

\newtheorem{theorem}{Theorem}[section]
\newtheorem{lemma}[theorem]{Lemma}

\newtheorem{defi}[theorem]{Definition}

\newtheorem{corollary}[theorem]{Corollary}
\newtheorem{remark}[theorem]{Remark}

\renewcommand{\theequation}{\thesection .\arabic{equation}}
\let\subs\subsection
\renewcommand\subsection{\setcounter{equation}{0}
\gdef\theequation{\thesubsection \arabic{equation}}\subs}
\let\sect\section
\renewcommand\section{\setcounter{equation}{0}
\gdef\theequation{\thesection .\arabic{equation}}\sect}

\newcommand{\IR}{{\mathbb{R}}}

\newcommand{\IZ}{{\mathbb{Z}}}
\newcommand{\zv}{\IZ^\nu}
\newcommand{\rv}{\IR^\nu}

\newcommand{\be}{\begin{equation}}
\newcommand{\ee}{\end{equation}}
\newcommand{\nn}{\nonumber}

\newcommand{\ve}{\varepsilon}

\newcommand{\vp}{\varphi}
\newcommand{\ka}{\kappa}

\newtheorem*{thma}{Theorem A}
\newtheorem*{thmb}{Theorem B}

\def\wmap{\omega = (\omega_1,\dots, \omega_\nu)}
\def\zero{{(0)}}

\def\one{{(1)}}

\begin{document}

\title[The KdV Equation with Quasi-Periodic Initial Data] {ON THE EXISTENCE AND UNIQUENESS OF GLOBAL SOLUTIONS FOR THE KDV EQUATION WITH QUASI-PERIODIC INITIAL DATA}

\author{David Damanik \and Michael Goldstein}

\address{Department of Mathematics at Rice University, 6100 S. Main St. Houston TX 77005-1892, U.S.A.}

\email{damanik@rice.edu}

\address{Department of Mathematics, University of Toronto, Bahen Centre, 40 St. George St., Toronto, Ontario, CANADA M5S 2E4}

\email{gold@math.toronto.edu}

\thanks{The first author was partially supported by a Simons Fellowship and NSF grants DMS--0800100, DMS--1067988, DMS--1361625. The second author was partially supported by a Guggenheim Fellowship and an NSERC grant.}

\begin{abstract}
We consider the KdV equation
$$
\partial_t u +\partial^3_x u +u\partial_x u=0
$$
with quasi-periodic initial data whose Fourier coefficients decay exponentially and prove existence and uniqueness, in the class of functions which have an expansion with exponentially decaying Fourier coefficients, of a solution on a small interval of time, the length of which depends on the given data and the frequency vector involved. For a Diophantine frequency vector and for small quasi-periodic data (i.e., when the Fourier coefficients obey $|c(m)| \le \ve \exp(-\kappa_0 |m|)$ with $\ve > 0$ sufficiently small, depending on $\kappa_0 > 0$ and the frequency vector), we prove global existence and uniqueness of the solution. The latter result relies on our recent work \cite{DG} on the inverse spectral problem for the quasi-periodic Schr\"{o}dinger equation.
\end{abstract}

\maketitle

\section{The Main Results}

The Korteweg-de Vries equation (KdV)
\begin{equation} \label{eq:I.KdV}
\partial_t u + \partial_x^3 u + u \partial_x u = 0
\end{equation}
is a non-linear equation arising in various physical systems. It was derived in the late 19th century to describe the propagation of long waves. In the 1960's, physicists found applications of the equation to a number of other models; compare Gardner, Morikawa~\cite{GaMo}, Zabusky~\cite{Za}, see also the review by Jeffrey and Kakutani~\cite{JeKa} featuring a number of models leading to KdV and its more general counterparts. It was realized approximately at the same time that in addition to a clear relevance to physical processes, the equation possesses a fundamental feature --- it has plenty of integrals (conservation laws). The search for these integrals by Zabusky, Miura, Kruskal culminated in the fundamental discovery by Miura, Gardner, Kruskal~\cite{MGK} who observed that if $u(t,x)$ is a solution of the KdV equation \eqref{eq:I.KdV}, then the eigenvalues of the Sturm-Liouville operator
\begin{equation} 
H_t = -\frac{d^2}{dx^2} - \frac{u}{6}, \quad x \in \mathbb{R}
\end{equation}
do not depend on $t$. The same fact was established in \cite{GGKM} in the periodic setting where the operator $H_t$ is defined via periodic boundary conditions on a given interval. Lax, in his classical work \cite{Lax}, explained how one can see that the operators $H_t$ are actually unitarily conjugate and for this reason their spectra naturally do not depend on $t$. This work established the connection between KdV and the inverse spectral problem for Sturm-Liouville operators. At that time the latter had been extensively developed mainly due to the fundamental works by Gelfand, Levitan and Marchenko; see the monographs \cite{LeSa}, \cite{Mar}. For a modern development of the theory on a finite interval, see the monograph of P\"{o}schel and Trubowitz ~\cite{PoTr}. The final fundamental fact which defined the interconnection of the KdV equation with different domains of mathematics was observed by Gardner \cite{Ga} and by Faddeev and Zakharov ~\cite{FZ}: in the periodic setting it can be viewed as an infinite dimensional completely integrable Hamiltonian system. All these fundamental discoveries led to a very rich development of the mathematical theory of the KdV equation and also of a number of other integrable non-linear systems of differential and difference equations.

The theory allowed one to solve three classes of initial data, fast decaying, periodic, and finite-gap quasi-periodic. The periodic and finite gap quasi-periodic theory is especially rich due to strong connections with the the theory of completely integrable dynamical systems and algebraic geometry; compare, for example, McKean-van Moerbeke~\cite{McKvM}, McKean-Trubowitz~\cite{McKT}, Novikov ~\cite{No}, Dubrovin~\cite{Du}, \cite{Du2}, Dubrovin-Matveev-Novikov \cite{DMN}, and Flaschka-McLaughlin~\cite{FlMcL}.

The development of the theory for other classes of initial data seems to encounter serious difficulties. One possibility naturally comes from the concepts of dynamical systems, to study small Hamiltonian perturbations of periodic KdV via KAM theory. For periodic KdV, an infinite dimensional version of the Liouville Theorem on completely integrable systems was established due to the above-mentioned fundamental works by McKean and van Moerbeke~\cite{McKvM}, McKean and Trubowitz~\cite{McKT} and the remarkable work by Flaschka and McLaughlin ~\cite{FlMcL} who found the action variables. For that matter the set-up for a KAM theory is absolutely natural. This beautiful program was developed further through the work of Kuksin~\cite{Ku1}, who found the crucial ingredient needed, and also works by Kuksin and P\"{o}schel ~\cite{KuPo} and Kappeler and Makarov ~\cite{KaMa}. See the monographs \cite{Ku2}, \cite{KP} for a presentation of this theory.

Quasi-periodic initial data, such as
$$
u_0(x) = \cos (x) + \cos(\omega x)
$$
with irrational $\omega$, are naturally of interest, but present a major challenge (when they are not of finite-gap type). The integration of the KdV equation with quasi-periodic initial data was included in Problem~1 in Deift's list of problems in random matrix theory and the theory of integrable systems; compare \cite{De}. Theorems~A and B, stated and proved below, represent progress regarding this problem.

We consider the KdV equation
\begin{equation} \label{eq:1.KdV}
\partial_t u +\partial_x^3 u + u\partial_x u=0
\end{equation}
\textit{with the initial data}
\begin{equation}\label{eq:1-1KdVip}
u_0(x) =\sum_{n \in \zv}\, c(n) e^{i x n \omega},
\end{equation}
$\wmap \in \rv$, $n\omega = \sum n_j \omega_j$. First we study the initial value problem. This is the approach which comes from PDE. The initial value problem for \eqref{eq:1.KdV} has been extensively studied via PDE methods: the works by Bona and Smith ~\cite{BoS}, Saut and Temam~\cite{SaT}, Kato~\cite{Kat}, Kappeler~\cite{Ka}, Kenig, Ponce, Vega~\cite{KePoVe}, \cite{KePoVe1}, Bourgain ~\cite{Bo}, \cite{Bo1}, Colliander, Keel, Staffilani, Takaoka, Tao~\cite{CKSTT} represent remarkable progress in the
study of the local existence problem for low regularity data for both $\mathbb{R}$ and $\mathbb{T}$. The work of Kappeler and Topalov~\cite{KaT} uses the inverse spectral problem approach and establishes almost optimal
results in the $\mathbb{T}$ setting, the solution exists and is continuous for data from the Sobolev space $H^\beta$ with $\beta \ge -1$.

To this end we would also like to mention a connection involving the KdV equation with very low regularity and highly random initial data like realizations of white noise. It was communicated by J.~Quastel to the second author that there seems to be a very fine relation between the scaling limits of such solutions of the KdV equation and the scaling limits of solutions of equations such as the stochastic Burgers or Kardar-Parisi-Zhang equations.

The existence of global solutions, in both the $\mathbb{R}$ and $\mathbb{T}$ settings, follows from the local result due to  the complete integrability. For instance the functionals
\begin{equation} \label{eq:1.KdVintegrals}
\begin{split}
\Phi_2(u) & = \int  u^2 \, dx, \\ \Phi_3(u) & = \int ((\partial_x u)^2-\frac{u^3}{3}) \, dx, \\ \Phi_4(u) & = \int ((\partial_x^2 u)^2-\frac{5}{3}u (\partial_x u)^2+\frac{5}{34}u^4) \, dx
\end{split}
\end{equation}
are preserved by the KdV dynamics. The integration domain here is either the whole real line $\mathbb{R}$, or the circle $\mathbb{T}$ in the periodic setting. Using these integrals one can verify that a solution with initial data from $H^2$ stays in $H^2$, that is, the problem is well-posed in this class. Moreover, the $H^2$ norm of the solution stays bounded. This implies the existence of a global solution. The analysis for lower regularity data is of course a much harder problem which was solved in the works we cited.

The methods of the work we cited do not apply automatically to quasi-periodic data \eqref{eq:1-1KdVip}. We use the approach by Kenig, Ponce, Vega; see \cite{KePoVe}. Using this approach, Bourgain introduced his norm-projection method and established the existence of global solutions of the KdV equation with local $L^2$ periodic data; see \cite{Bo}. Tsugawa modified the method of Bourgain to the case when the function $u_0$ has the form $u_0(x) = \sum_{m \in \mathbb{Z}^\nu} c(m) e^{2 \pi i x m \omega}$ with Diophantine $\omega$ and with $|c(m)|\le B_0 (1+|m|)^{-A}$, where $A$ is large, and proved local well-posedness; see \cite{Tsu}. However, the method does not control the norm used. For this reason the method does not guarantee the existence of a global solution.

\smallskip

From the perspective of PDE methods, the main difficulty with quasi-periodic initial data is in the complicated nature of the conservation laws. For instance, the following appear to be natural candidates,
\begin{equation} \label{eq:1.KdVintegralsqp}
\begin{split}
\Phi_2(u) & = \lim_{A \to \infty} \frac{1}{2A} \int_{-A}^A u^2 \, dx, \\
\Phi_3(u) & = \lim_{A \to \infty} \frac{1}{2A} \int_{-A}^A ((\partial_x u)^2-\frac{u^3}{3}) \, dx, \\
\Phi_4(u) & = \lim_{A \to \infty} \frac{1}{2A} \int_{-A}^A ((\partial_x^2 u)^2-\frac{5}{3}u (\partial_x u)^2+\frac{5}{34}u^4) \, dx.
\end{split}
\end{equation}
However, there is no direct way to ``integrate by parts'' in order to verify that indeed these averages are preserved by the KdV dynamics. On the other hand, the Lax conjugation does work for the operators on the whole real line $\mathbb{R}$, and the spectrum is therefore preserved. \textit{In the current work we use the conservation of the spectrum to derive global solutions with quasi-periodic initial data from the local existence result.}

\medskip

The spectral problem for the quasi-periodic Sturm-Liouville operator and its discrete counterpart has been extensively studied in the past 40 years, starting with the ground-breaking work of Dinaburg and Sinai~\cite{DiSi}. They discovered the existence of so-called Bloch-Floquet solutions
\[
\psi(k, x) = \sum_{n \in \zv} \vp(n; k) e^{i x (n\omega+k)}
\]
for a large set of $k$'s, provided that the potential is analytic and small. In \cite{El}, Eliasson completely solved the problem by proving in the same setting the existence of Bloch-Floquet solutions for almost all $k$, and also the fact that the spectrum is purely absolutely continuous. For an extensive discussion of the spectral theory of discrete quasi-periodic Schr\"{o}dinger operators, see the monograph \cite{Bo2} by Bourgain. The spectrum in the quasi-periodic case is known to typically have a dense set of gaps. Furthermore, the structure of the eigenfunctions exhibits a phase transition as the magnitude of the potential grows. All these features present serious difficulties in the development of the theory. There is no inverse spectral theory available at this point. Recently the authors of the present paper developed a method which allowed them to obtain two-sided estimates relating the size of the gaps of the spectrum and the magnitude of the Fourier coefficients of the potential, see \cite{DG}. The method works for small analytic potentials only. In the current paper we therefore need to establish the well-posedness of the solution of the KdV equation in this class of potentials.

Throughout this paper we use the notation $|x|$ for the $\ell^1$-norm on $\mathbb{R}^k$, that is,
\[
|x|=\sum_j |x_j|,\quad x=(x_1,\dots,x_k)\in \mathbb{R}^k.
\]

\bigskip

\begin{thma}
Assume that $|c(n)| \le B_0 \exp(-\kappa |n|)$, where $B_0, \kappa > 0$ are constants. There exists $t_0 > 0$ such that for $0 \le t < t_0$, $x \in \mathbb{R}$, one can define a function
$$
u(t,x) = \sum_{n \in \zv}\, c(t,n) e^{i x n \omega},
$$
with $|c(t,n)| \le 2 B_0 \exp(-\frac{\kappa}{2} |n|)$, which obeys equation \eqref{eq:1.KdV} with the initial condition $u(0,x)=u_0(x)$.

Furthermore, assume $n \omega \neq 0$ for every $n\neq 0$. If
$$
v(t,x) = \sum_{n \in \zv}\, h(t,n) e^{i x n \omega},
$$
with $|h(t,n)|\le B \exp(-\rho |n|)$ for some constants $B , \rho > 0$, obeys equation \eqref{eq:1.KdV} with the initial condition $v(0,x) = u_0(x)$, then there exists $t_1 > 0$ such that $v(t,x) = u(t,x)$ for $0 \le t < t_1$, $x \in \mathbb{R}$.
\end{thma}

The derivation of Theorem~A proceeds, as pointed out above, via the Kenig-Ponce-Vega approach. We do not follow the method of Bourgain, but rather do an explicit combinatorial analysis of the iteration of the integral transformation defined via the approach by Kenig-Ponce-Vega, applied to the initial data. Let us mention that there are certain similarities between our derivation and the one in the work \cite{Ch} by Christ. The main difference is that we need to keep the exponential decay estimates in check.

Although the derivation is rather involved, the exponential decay of the Fourier coefficients plays out nicely with the combinatorial growth factors produced in the iterations. No small denominator problem enters the estimation. Due to this, our derivation of local existence and uniqueness does not rely on Diophantine properties of $\omega$, as opposed to Tsugawa's derivation.

Most importantly, the Fourier coefficients of the solution in Theorem~A obey an \textbf{exponential decay estimate}. This estimate is absolutely crucial in our derivation of the global solution. This is because the inverse spectral analysis in our work \cite{DG} applies only to the class of quasi-periodic Sturm-Liouvile operators with analytic potentials. In order to develop a similar analysis for the class of smooth potentials one will have to overcome a number of hard technical problems.

On the other hand, the spectral analysis of the quasi-periodic Sturm-Liouville operator relies heavily on the Diophantine condition on the vector $\omega$. Actually, almost all typical spectral results known for quasi-periodic operators assume Diophantine conditions. Furthermore, it is known that a Liouville frequency can change completely the structure of the spectrum and the eigenfunctions. It is also known that low regularity potentials can cause similar kinds of phenomena.
From this perspective, small quasi-periodic analytic potentials, for which there is no phase transition in the structure of the spectrum and eigenfunctions, appear to be a natural class of initial data for KdV.

\begin{thmb}
Assume that the vector $\omega$ satisfies the following Diophantine condition:
$$
|n\omega| \ge 2 \pi a_0 |n|^{-b_0}, \quad n \in \zv \setminus \{0\}
$$
with some $0 < a_0 < 1$, $\nu-1 < b_0 < \infty$. Given $\kappa_0 > 0$, there exists $\ve^\one = \ve^\one(\ka_0, a_0, b_0) > 0$ such that if $|c(n)| \le \ve^\one \exp(-\kappa_0 |n|)$, then for $0 \le t < \infty$, $x \in \mathbb{R}$, one can define a function
$$
u(t,x) = \sum_{n \in \zv}\, c(t,n) e^{i x n \omega}
$$
with $|c(t,n)| \le (\ve^\one)^{1/4} \exp(-\frac{\kappa_0}{8} |n|)$, which obeys equation \eqref{eq:1.KdV} with the initial condition $u(0,x) = u_0(x)$.

Moreover, let
$$
v(t,x) = \sum_{n \in \zv}\, h(t,n) e^{i x n \omega},
$$
with $|h(t,n)| \le B \exp(-\rho |n|)$ for some constants $B , \rho > 0$. If $v$ obeys \eqref{eq:1.KdV} for $t \ge 0$, $x \in \mathbb{R}$, with the same initial condition $v(0,x) = u_0(x)$, then $v(t,x) = u(t,x)$ for $t \ge 0$, $x \in \mathbb{R}$.
\end{thmb}

\begin{remark}\label{rem.egorova}
Let us mention that in \cite{Eg}, Egorova proved the existence of global solutions of the KdV equation with certain limit-periodic initial data $u_0(x)$, for which the spectrum of the Sturm-Liouville operator
\[
H_0 = -\frac{d^2}{dx^2} + u_0
\]
obeys certain hierarchical geometric conditions. Her result applies in particular to the limit-periodic functions studied by Pastur and Tkachenko in \cite{PT}. Note that a non-periodic function cannot be simultaneously limit-periodic and quasi-periodic, and hence there is no overlap between our result and hers.
\end{remark}
%

Before we prove the main results, we first develop in Section~\ref{sec.2} estimates on the iterations of the integral equation in question. The equation itself is the basic one in the Kenig-Ponce-Vega approach. The equation implies a system of nonlinear integral equations for the Fourier coefficients of the solution. The iteration of the system of equations produces summations of certain weights over complicated trees. The weights themselves are multiplicative functions of the tree branches when calculated along each linkage in the tree. The combinatorics is rather involved and the estimates require certain re-summations via isomorphisms between the tree parts and standard lattices. Through a series of technical lemmas we build up toward the main estimate of Section~\ref{sec.2}, which is Corollary~\ref{cor:cestimatesBkRE}. It states the exponential
convergence for the solutions of the iterated system of integral equations in question. In Section~\ref{sec.3} we invoke Corollary~\ref{cor:cestimatesBkRE} to show that the limit Fourier coefficients, which exist due to this corollary, indeed define a local solution of the KdV equation with given analytic quasi-periodic data. Moreover, the estimates developed in Section~\ref{sec.2} show that if the initial data was ``small,'' then so is the local solution. The estimates of Section~\ref{sec.2} also allow us to show that the local solution is unique in the class of analytic quasi-periodic solutions. The derivation of the global solution relies heavily on the complete integrability and the main estimates of \cite{DG}. The issue here is that the Fourier coefficients of the local solution
$$
u(t,x) = \sum_{n \in \zv}\, c(t,n) e^{i x n \omega}
$$
in Theorem~A just obey the estimate $|c(t,n)| \le 2 B_0 \exp(-\frac{\kappa}{2} |n|)$. This estimate is a bit worse than the estimate for the Fourier coefficients of the initial data. One therefore faces the problem that after several consecutive applications of Theorem A, the existence interval for the local solution could shrink. Here the conservation of the spectrum comes to the rescue. The spectrum of the quasi-periodic Sturm-Liouville operator $H_t$ in \eqref{eq:1.KdV} is the same as the spectrum of $H_0$. The two-sided estimates from \cite{DG} show that, although the Fourier coefficients in question can change, the gap estimates for $H_0$ keep the coefficients in check. Due to this fact the existence interval for the local solution will not shrink much, which in turn implies that there is a global solution.

\section*{Acknowledgment}

We are grateful to the referee for carefully reading the manuscript and providing an extensive list of excellent comments and suggestions.

\section{Preliminary Lemmas}\label{sec.2}

Let us start with an elementary observation:

\begin{lemma}\label{lem:12Fourier1}
Let $g(t)$ be a continuous function of $t \in [0,t_0)$, $t_0>0$ and let $\alpha \in \mathbb{R}$. Then, the following identity holds for $t \in [0,t_0)$:
\begin{align*}
\partial_t \int_0^t \exp(i [(t - \tau) \alpha^3 + \alpha x]) g(\tau) \, d \tau & = g(t) \exp(i \alpha x) - \partial^3_x \int_0^t \exp(i [(t - \tau) \alpha^3 + \alpha x]) g(\tau) \, d \tau.
\end{align*}
\end{lemma}

\begin{proof}
This follows readily by explicit differentiation on both sides of the identities.
\end{proof}

The system of integral equations for the Fourier coefficients of the solution of KdV with quasi-periodic initial data,
which we mentioned in the introduction, follows from the differentiation formula in Lemma~\ref{lem:12Fourier1} combined with the idea of the Kenig-Ponce-Vega approach:

\begin{lemma}\label{lem:12Fourier2}
Let $c(t,n)$, $g(t,n)$ be continuous functions of $t \in [0, t_0)$, $n \in \mathbb{Z}^\nu$. Assume that
\begin{equation}\label{eq:12Fourierccond}
\sup_{t} \sum_{n \in \mathbb{Z}^\nu} (1+ |n|^3) (|c(t,n)| + |g(t,n)|) < \infty.
\end{equation}
Let $\omega \in \mathbb{R}^\nu$. Assume that the integral equations
\begin{equation}\label{eq:12Fourier2}
c(t,n) = c(0,n) \exp(i t(n\omega)^3) + \int_0^t \exp(i(t - \tau) (n\omega)^3) g(\tau,n) \, d \tau, \quad n \in \mathbb{Z}^\nu
\end{equation}
hold for $n\in\mathbb{Z}^\nu$. Then the function
$$
u(t,x) = \sum_{n \in \mathbb{Z}^\nu} c(t,n) \exp(ixn\omega)
$$
obeys the differential equation
$$
\partial_t u = -\partial^3_x u + v
$$
with $u(0,x) = u_0(x)$, where
\begin{align*}
u_0(x) & = \sum_{n \in \mathbb{Z}^\nu} c(0,n) \exp(i x n \omega), \\
v(t,x) & = \sum_{n \in \mathbb{Z}^\nu} g(t,n) \exp(i x n \omega).
\end{align*}
The functions $\partial_t u, -\partial^3_x u, v$ are continuous throughout  the domain $t\in[0,t_0)$, $x\in\mathbb{R}$.
\end{lemma}

\begin{proof}
Using the equations \eqref{eq:12Fourier2} and Lemma~\ref{lem:12Fourier1}, one obtains
\begin{align*}
\partial_t u & = \sum_{n \in \mathbb{Z}^\nu} \partial_t c(t,n) \exp(i x n \omega) \\
& = \sum_{n \in \mathbb{Z}^\nu} \partial_t \left[ c(0,n) \exp(i t(n \omega)^3) + \int_0^t \exp(i(t - \tau) (n \omega)^3) g(\tau,n) \, d \tau \right] \exp(i x n \omega) \\
& = \sum_{n \in \mathbb{Z}^\nu} \partial_t \left[ c(0,n) \exp(i t(n \omega)^3 + i x n \omega) + \int_0^t \exp(i[(t - \tau)(n \omega)^3 + x n \omega]) g(\tau,n) \, d \tau \right] \\
& = \sum_{n \in \mathbb{Z}^\nu} -\partial^3_x \left[ c(0,n) \exp(i t(n\omega)^3 + i x n \omega) \right] + g(t,n) \exp(2 \pi ixn\omega) - \partial^3_x \int_0^t \exp(i[(t - \tau) (n \omega)^3 + x n \omega]) g(\tau,n) \, d \tau \\
& = \sum_{n \in \mathbb{Z}^\nu} -\partial^3_x \left[ c(0,n) \exp(i t(n \omega)^3 + i x n \omega) + \int_0^t \exp(i[(t - \tau) (n \omega)^3 + x n \omega]) g(\tau,n) \, d \tau \right] + g(t,n) \exp(i x n \omega) \\
& = \sum_{n \in \mathbb{Z}^\nu} - \partial^3_x \left[ \Bigl( c(0,n) \exp(i t (n \omega)^3) + \int_0^t \exp(i(t - \tau) (n \omega)^3) g(\tau,n) \, d \tau \Bigr) \exp(i x n \omega) \right] + \sum_{n \in \mathbb{Z}^\nu} g(t,n) \exp(2 \pi ixn\omega) \\
& = \sum_{n \in \mathbb{Z}^\nu} - \partial^3_x c(t,n) \exp(i x n \omega) + \sum_{n \in \mathbb{Z}^\nu} g(t,n) \exp(i x n \omega) \\
& = -\partial^3_x u + v,
\end{align*}
as claimed. The absolute and uniform convergence of all series involved follows from condition \eqref{eq:12Fourierccond}.
\end{proof}

We want to iterate the system \eqref{eq:12Fourier2}. As was explained in the introduction, we have to start the iteration from an exponentially decaying collection of Fourier coefficients and keep this property in check. Let $c(n)$ be a function of $n \in \mathbb{Z}^\nu$ such that
\begin{equation}\label{eq:decay}
|c(n)| \le B \exp(-\kappa|n|),
\end{equation}
where $B, \kappa > 0$ are constants. For convenience \textbf{we always assume that $\kappa\le 1$}. Let $\omega \in \mathbb{R}^\nu$. Set
\begin{equation}\label{eq:c0def}
c_0(t,n) = c(n) \exp(i t(n\omega)^3),
\end{equation}
and for $k = 1,2,\dots$,
\begin{equation}\label{eq:12FourierFdefi1a}
c_k(t,n) = c(n) \exp(i t(n\omega)^3) -\frac{i n \omega}{2} \int_0^t \exp(i(t - \tau) (n \omega)^3) \sum_{m_1, m_2 \in \mathbb{Z}^\nu : m_1 + m_2 = n} c_{k-1} (\tau,m_1) c_{k-1} (\tau,m_2) \, d \tau,
\end{equation}
$n \in \mathbb{Z}^\nu$.

\begin{remark}\label{rem.cont}
$(1)$ We will show by induction that for some $t_0 > 0$, the functions $c_k(t,n)$ are well-defined and continuous for $0 \le t \le t_0$ and that the series in \eqref{eq:12FourierFdefi1a} converges absolutely and uniformly on the interval $0 \le t \le t_0$.

$(2)$ The structure of the summation in \eqref{eq:12FourierFdefi1a} suggests to label the terms of the iterated equation via points from a tree whose branches originate at points on the lattice $\mathbb{Z}^\nu$, with the branch split over the condition $m_1 + m_2 = constant$. Below we introduce inductively the branches $\gamma$ and the ``cumulative degree'' $\mathfrak{F}(\gamma)$ calculated along a given branch. To bookkeep the terms of the iterated equation, we also need to attach to each branch an appropriate lattice $\mathbb{Z}^K$. Ultimately, we define certain weights which enable us to develop an estimation technique for the iterated equations. Although the definition of these objects looks rather complicated, a closer look at the integral equation structure shows that they arise naturally from it via induction over the number of iterations of the equation.
\end{remark}

Set
\begin{equation}\label{eq:12Bernoulli1}
\Gamma^{(1)} = \{0,1\}, \quad \Gamma^{(k)} = \{0\} \cup \Gamma^{(k-1)} \times \Gamma^{(k-1)}, \quad k = 2, \dots,
\end{equation}
and
\begin{align}\label{eq:12Bernoulli12}
\mathfrak{l}(\gamma) & = \begin{cases} \gamma & \text{if $\gamma \in \Gamma^{(1)}$}, \\ 0 & \text{if $\gamma = 0 \in \Gamma^{(k)}$, $k \ge 2$}, \\ \mathfrak{l}(\gamma^{(k-1)}_1) + \mathfrak{l}(\gamma^{(k-1)}_2) + 1 & \text{if $k \ge 2$, $\gamma = (\gamma^{(k-1)}_1,\gamma^{(k-1)}_2) \in \Gamma^{(k-1)} \times \Gamma^{(k-1)}$}, \end{cases},
\end{align}
\begin{align}\label{eq:12Bernoulli12a}
\mathfrak{d}(\gamma) & = \begin{cases} 1 & \text{if $\gamma = 0 \in \Gamma^{(k)}$}, \\ 2 & \text{if $\gamma = 1 \in \Gamma^{(1)}$}, \\ \mathfrak{d}(\gamma^{(k-1)}_1) + \mathfrak{d}(\gamma^{(k-1)}_2) & \text{if $k \ge 2$, $\gamma = (\gamma^{(k-1)}_1, \gamma^{(k-1)}_2) \in \Gamma^{(k-1)} \times \Gamma^{(k-1)}$}, \end{cases},
\end{align}
\begin{align}\label{eq:12Bernoulli12b}
\mathfrak{F}(\gamma) & = \begin{cases} 1 & \text{if $\gamma \in \Gamma^{(1)}$ or if $\gamma = 0 \in \Gamma^{(k)}$, $k\ge 2$}, \\ \mathfrak{l}(\gamma) \mathfrak{F}(\gamma^{(k-1)}_1) \mathfrak{F}(\gamma^{(k-1)}_2) & \text{if $k \ge 2$, $\gamma = (\gamma^{(k-1)}_1,\gamma^{(k-1)}_2) \in \Gamma^{(k-1)} \times \Gamma^{(k-1)}$.} \end{cases}
\end{align}

Note that just from the definition,
\[ \mathfrak{d}(\gamma)=\mathfrak{l}(\gamma)+1 \].

Furthermore, set
\begin{align}
\mathfrak{M}^{(k,\gamma)} & = \begin{cases} \mathbb{Z}^\nu & \text{if $\gamma = 0 \in \Gamma^{(k)}$}, \\ \mathbb{Z}^\nu \times \mathbb{Z}^\nu & \text{if $\gamma \in \Gamma^{(1)}$, $\gamma = 1$}, \\ \mathfrak{M}^{(k-1,\gamma^{(k-1)}_1)} \times \mathfrak{M}^{(k-1,\gamma^{(k-1)}_2)} & \text{if $\gamma \in \Gamma^{(k)}$,$k\ge 2$, $\gamma = (\gamma^{(k-1)}_1,\gamma^{(k-1)}_2) \in \Gamma^{(k-1)} \times \Gamma^{(k-1)}$}, \end{cases} \\
\mathfrak{B}^{(k,\gamma)} & = \begin{cases} \mathbb{Z} & \text{if $\gamma = 0 \in \Gamma^{(k)}$}, \\ \mathbb{Z} \times \mathbb{Z} & \text{if $\gamma \in \Gamma^{(1)}$, $\gamma = 1$}, \\ \mathfrak{B}^{(k-1,\gamma^{(k-1)}_1)} \times \mathfrak{B}^{(k-1,\gamma^{(k-1)}_2)} & \text{if $\gamma \in \Gamma^{(k)}$, $k\ge 2$, $\gamma = (\gamma^{(k-1)}_1,\gamma^{(k-1)}_2) \in \Gamma^{(k-1)} \times \Gamma^{(k-1)}$}, \end{cases} \\
\mu(\mathfrak{m}) & = \sum_j m_j, \quad \text{where $\mathfrak{m} = (m_1,\dots ,m_N)$, $m_j \in \mathbb{Z}^\lambda$},
\end{align}
and for $m^{(k)} \in \mathfrak{M}^{(k,\gamma)}$ and $t > 0$,
\begin{align}\label{eq:12Bernoulli12c}
|m^{(k)}|& = \begin{cases} |m| & \text{if $\gamma = 0 \in \Gamma^{(k)}$, $m^{(k)}=m\in\mathbb{Z}^\nu$}, \\ |m_1|+|m_2| & \text{if $\gamma \in \Gamma^{(1)}$, $\gamma = 1$, $m^{(1)} = (m_1,m_2)$}, \\ | m^{(k-1)}_1|+|m^{(k-1)}_2| & \text{if $\gamma \in \Gamma^{(k)}$, $\gamma = (\gamma^{(k-1)}_1,\gamma^{(k-1)}_2) \in \Gamma^{(k-1)} \times \Gamma^{(k-1)}$},\\ & \quad \text{ $m^{(k)} = (m^{(k-1)}_1,m^{(k-1)}_2) \in \mathfrak{M}^{(k-1,\gamma^{(k-1)}_1)} \times \mathfrak{M}^{(k-1,\gamma^{(k-1)}_2)}$}, \end{cases},
\end{align}
\begin{align}\label{eq:12Bernoulli12d}
\mathfrak{f}(m^{(k)}) & = \begin{cases} 1 & \text{if $\gamma = 0 \in \Gamma^{(k)}$, $m^{(k)} \in \mathfrak{M}^{(k,\gamma)}$}, \\ -\frac{i \mu(m^{(1)})\omega}{2} & \text{if $\gamma \in \Gamma^{(1)}$, $\gamma = 1$, $m^{(1)} \in \mathfrak{M}^{(1,\gamma)}$}, \\
-\frac{i \mu(m^{(k)}) \omega}{2} \mathfrak{f}(m^{(k-1)}_1) \mathfrak{f}(m^{(k-1)}_2) & \text{if $\gamma \in \Gamma^{(k)}$, $k \ge 2$,} \\ & \gamma = (\gamma^{(k-1)}_1,\gamma^{(k-1)}_2) \in \Gamma^{(k-1)} \times \Gamma^{(k-1)}, \\ & \text{$m^{(k)} = (m^{(k-1)}_1,m^{(k-1)}_2) \in \mathfrak{M}^{(k-1,\gamma^{(k-1)}_1)} \times \mathfrak{M}^{(k-1,\gamma^{(k-1)}_2)}$,} \end{cases}
\end{align}
\begin{align}
\mathfrak{P}(m^{(k)}) & = \begin{cases} 1 & \text{if $m^{(k)} \in \mathfrak{M}^{(k,\gamma)}$, $\gamma = 0 \in \Gamma^{(k)}$}, \\ |\mu(m^{(1)})| & \text{if $k = 1$, $m^{(k)} \in \mathfrak{M}^{(k,\gamma)}$, $\gamma = 1 \in \Gamma^{(k)}$}, \\ |\mu(m^{(k)})| \mathfrak{P}(m^{(k-1)}_1) \mathfrak{P}(m^{(k-1)}_2) & \text{if $m^{(k)} \in \mathfrak{M}^{(k,\gamma)}$, $k \ge 2$}, \\ & \gamma = (\gamma^{(k-1)}_1,\gamma^{(k-1)}_2) \in \Gamma^{(k-1)} \times \Gamma^{(k-1)}, \\ & m^{(k)} = (m^{(k-1)}_1, m^{(k-1)}_2) \in \mathfrak{M}^{(k-1,\gamma^{(k-1)}_1)} \times \mathfrak{M}^{(k-1,\gamma^{(k-1)}_2)}, \end{cases} \label{eq:12FourierB2a} \\
I(t,m^{(k)}) & = \begin{cases} \exp(i t(\mu(m^{(k)})\omega)^3) \quad \text{if $\gamma = 0 \in \Gamma^{(k)}$, $m^{(k)} \in \mathfrak{M}^{(k,\gamma)}$}, \\ \int_0^t \exp(i(t - \tau)(\mu(m^{(1)})\omega)^3) \exp(i \tau(m_1\omega)^3) \exp(i \tau(m_2\omega)^3) \, d \tau \\ \quad \qquad \qquad \qquad \qquad \text{if $\gamma \in \Gamma^{(1)}$, $\gamma = 1$, $m^{(1)} = (m_1,m_2) \in \mathfrak{M}^{(1,\gamma)}$}, \\ \int_0^t \exp(i(t - \tau) (\mu(m^{(k)})\omega)^3) I(\tau,m^{(k-1)}_1) \times I(\tau,m^{(k-1)}_2) \, d \tau \\ \quad \qquad \qquad \qquad \qquad \text{if $\gamma \in \Gamma^{(k)}$, $k \ge 2$, $\gamma = (\gamma^{(k-1)}_1,\gamma^{(k-1)}_2) \in \Gamma^{(k-1)} \times \Gamma^{(k-1)}$,} \\ \quad \qquad \qquad \qquad \qquad \text{$m^{(k)} = (m^{(k-1)}_1,m^{(k-1)}_2) \in \mathfrak{M}^{(k-1,\gamma^{(k-1)}_1)} \times \mathfrak{M}^{(k-1,\gamma^{(k-1)}_2)}$}. \end{cases}
\label{eq:12FourierB2}
\end{align}

\begin{remark}\label{rem.cont123}
The functions $\mathfrak{f}^{(k,\gamma)}, (m^{(k)}) I^{(k,\gamma)}(t,m^{(k)})$ represent the basic terms in the expansion for the iterated Fourier coefficients. Namely, in part $(2)$ of Corollary~\ref{cor:12system1cestimates} we show that
$$
c_k(t,n) = \sum_{\gamma \in \Gamma^{(k)}} \sum_{m^{(k)} \in \mathfrak{M}^{(k,\gamma)} : \mu(m^{(k)}) = n} \mathfrak{C} (m^{(k)}) \mathfrak{f}^{(k,\gamma)} (m^{(k)}) I^{(k,\gamma)}(t,m^{(k)}).
$$
with appropriate coefficients $\mathfrak{C} (m^{(k)})$. To explain the convergence of the series here we first develop estimates for these functions. After that we prove the expansion.
\end{remark}

In the next lemma we develop the basic rules for the evaluation of the functions introduced.

\begin{lemma}\label{lem:12system1}
The following statements hold:

$(1)$
\begin{align*}
\big| I(t,m^{(1)}) \big| & \le 1 \quad \text{if $\gamma \in \Gamma^{(1)}$, $\gamma = 0$, $m^{(1)} \in \mathfrak{M}^{(1,\gamma)}$}, \\
\big| I(t,m^{(1)}) \big| & \le t \quad \text{if $\gamma \in \Gamma^{(1)}$, $\gamma = 1$, $m^{(1)} \in \mathfrak{M}^{(1,\gamma)}$}, \\
\big| I(t,m^{(k)}) \big| & \le 1 \quad \text{if $\gamma = 0 \in \Gamma^{(k)}$, $\gamma = 0$, $m^{(k)} \in \mathfrak{M}^{(k,\gamma)}$, $k \ge 2$}, \\
\big| I(t,m^{(k)}) \big| & \le \frac{t^{\mathfrak{l}(\gamma)}}{\mathfrak{F}(\gamma)} \quad \text{if $\gamma \in \Gamma^{(k)}$, $k \ge 2$, $\gamma = (\gamma^{(k-1)}_1, \gamma^{(k-1)}_2) \in \Gamma^{(k-1)} \times \Gamma^{(k-1)}$}.
\end{align*}

$(2)$
\begin{align*}
|\mathfrak{f}(m^{(1)})| & = 1 \quad \text{if $\gamma \in \Gamma^{(1)}$, $\gamma = 0$, $m^{(1)} \in \mathfrak{M}^{(1,\gamma)}$}, \\
|\mathfrak{f}(m^{(1)})| & \le |\omega| |\mu(m^{(1)})| \quad \text{if $\gamma \in \Gamma^{(1)}$, $\gamma = 1$, $m^{(1)} \in \mathfrak{M}^{(1,\gamma)}$}, \\
|\mathfrak{f}(m^{(k)})| & = 1 \quad \text{if $\gamma = 0 \in \Gamma^{(k)}$, $k \ge 2$, $m^{(k)} \in \mathfrak{M}^{(k,\gamma)}$}, \\
|\mathfrak{f}(m^{(k)})| & \le |\omega|^{\mathfrak{l}(\gamma)} \mathfrak{P}(m^{(k)}) \quad \text{if $\gamma \in \Gamma^{(k-1)} \times \Gamma^{(k-1)}$, $m^{(k)} \in \mathfrak{M}^{(k,\gamma)}$}.
\end{align*}
\end{lemma}

\begin{proof}

$(1)$ Using the definitions \eqref{eq:c0def}--\eqref{eq:12FourierB2} and induction, one obtains
\begin{align*}
\big| I(t,m^{(1)}) \big| & \le 1 \quad \text{if $\gamma \in \Gamma^{(1)}$, $\gamma = 0$, $m^{(1)} \in \mathfrak{M}^{(1,\gamma)}$}, \\
\big| I(t,m^{(1)}) \big| & \le t \quad \text{if $\gamma \in \Gamma^{(1)}$, $\gamma = 1$, $m^{(1)} \in \mathfrak{M}^{(1,\gamma)}$}, \\
\big| I(t,m^{(k)}) \big| & \le 1 \quad \text{if $\gamma = 0 \in \Gamma^{(k)}$, $\gamma = 0$, $m^{(k)} \in \mathfrak{M}^{(k,\gamma)}$, $k \ge 2$}, \\
\big| I(t,m^{(k)}) \big| & \le \int_0^t \big| I (\tau,m^{(k-1)}_1) \big| \big| I(\tau,m^{(k-1)}_2) \big| \, d \tau \\
& \le \int_0^t \frac{\tau^{\mathfrak{l}(\gamma^{(k-1)}_1)}}{\mathfrak{F}(\gamma^{(k-1)}_1)} \frac{\tau^{\mathfrak{l}(\gamma^{(k-1)}_2)}}{\mathfrak{F}(\gamma^{(k-1)}_2)} \, d \tau = \int_0^t \frac{\tau^{\mathfrak{l}(\gamma^{(k-1)}_1) + \mathfrak{l}(\gamma^{(k-1)}_2)}}{\mathfrak{F}(\gamma^{(k-1)}_1) \mathfrak{F}(\gamma^{(k-1)}_2)} \, d \tau \\
& = \frac{t^{\mathfrak{l}(\gamma^{(k-1)}_1) + \mathfrak{l}(\gamma^{(k-1)}_2)+1}}{(\mathfrak{l}(\gamma^{(k-1)}_1) + \mathfrak{l}(\gamma^{(k-1)}_2)+1) \mathfrak{F}(\gamma^{(k-1)}_1) \mathfrak{F}(\gamma^{(k-1)}_2)} \\
& = \frac{t^{\mathfrak{l}(\gamma)}}{\mathfrak{F}(\gamma)} \qquad \text{if $\gamma \in \Gamma^{(k)}$, $k \ge 2$, $\gamma = (\gamma^{(k-1)}_1, \gamma^{(k-1)}_2) \in \Gamma^{(k-1)} \times \Gamma^{(k-1)}$}, \\
\end{align*}
as claimed.

$(2)$ Similarly,
\begin{align*}
|\mathfrak{f}(m^{(1)})| & = 1 \quad \text{if $\gamma \in \Gamma^{(1)}$, $\gamma = 0$, $m^{(1)} \in \mathfrak{M}^{(1,\gamma)}$}, \\
|\mathfrak{f}(m^{(1)})| & \le |\omega| |\mu(m^{(1)})| \quad \text{if $\gamma \in \Gamma^{(1)}$, $\gamma = 1$, $m^{(1)} \in \mathfrak{M}^{(1,\gamma)}$}, \\
\mathfrak{f}(m^{(k)}) & = 1 \quad \text{if $\gamma = 0 \in \Gamma^{(k)}$, $k \ge 2$, $m^{(k)} \in \mathfrak{M}^{(k,\gamma)}$}, \\
|\mathfrak{f}(m^{(k)})| & \le |\omega| |\mu(m^{(k)})| |\mathfrak{f}(m^{(k-1)}_1)| |\mathfrak{f}(m^{(k-1)}_2)| \\
& \le |\omega| |\mu(m^{(k)})| |\omega|^{\mathfrak{l}(\gamma^{(k-1)}_1)} \mathfrak{P}(m^{(k-1)}_1) |\omega|^{\mathfrak{l}(\gamma^{(k-1)}_2)} \mathfrak{P}(m^{(k-1)}_2) = |\omega|^{\mathfrak{l}(\gamma)} \mathfrak{P}(m^{(k)}) \\
& \qquad \text{if $\gamma \in \Gamma^{(k)}$, $k \ge 2$, $\gamma = (\gamma^{(k-1)}_1, \gamma^{(k-1)}_2) \in \Gamma^{(k-1)} \times \Gamma^{(k-1)}$,} \\
& \qquad \text{$m^{(k)} = \{m^{(k-1)}_1, m^{(k-1)}_2\} \in \mathfrak{M}^{(k-1,\gamma^{(k-1)}_1)} \times \mathfrak{M}^{(k-1,\gamma^{(k-1)}_2)}$},
\end{align*}
as claimed.
\end{proof}

It is helpful to note here that when we apply the last lemma, the estimates in cases with $k\ge 2$ and $\gamma\neq 0$ subsume
the rest of the cases in the sense that the formal evaluation with $\mathfrak{l}(\gamma),\mathfrak{F}(\gamma), \mathfrak{P}(m^{(k)}) $, taken from \eqref{eq:12Bernoulli12}, \eqref{eq:12Bernoulli12b}, \eqref{eq:12FourierB2a} gives exactly the same result as for the rest of the cases.

\begin{remark}\label{rem.cont234}
Our next goal is to build up an estimation of the sums involving the functions $\mathfrak{f}^{(k,\gamma)}, (m^{(k)}) I^{(k,\gamma)}(t,m^{(k)})$. The main difficulty comes from the complicated combinatorics of the summation process.
As a matter of fact, the right estimate for the summation hinges on certain combinatorial identities, see the identity in part $(1)$ of Lemma~\ref{lem:elementarycalculus} and identity \eqref{eq:Pexpansionsumtatementest1aNNN} in the proof of part $(1)$
of Lemma~\ref{lem:PexpansionsumestA}. To employ effectively the identities in question, we ``change variables'' in the summations. Moreover, we do it twice. To prove the first identity we introduce certain isomorphisms of the lattices $\mathfrak{M}^{(k,\gamma)}$, $\mathfrak{B}^{(k,\gamma)}$ onto the standard lattice $\prod_{j=1}^{\mathfrak{d}(\gamma)} \mathbb{Z}^\nu$. To prove the second identity we introduce certain mappings of simplices of the standard lattice $\prod_{j=1}^{\mathfrak{d}(\gamma)} \mathbb{Z}^\nu$, see Remark~\ref{rem.cont345}.
\end{remark}

\begin{defi}\label{def:aux1}
$1$. We define inductively the isomorphism $\phi^{(k)}_\gamma : \mathfrak{M}^{(k,\gamma)} \rightarrow \prod_{j=1}^{\mathfrak{d}(\gamma)} \mathbb{Z}^\nu$ by setting
$$
\phi^{(k)}_\gamma(m^{(k)}) = \begin{cases} m & \text{if $\gamma = 0 \in \Gamma^{(k)}$ $m^{(k)} = m \in \mathbb{Z}^\nu = \mathfrak{M}^{(k,\gamma)}$}, \\ (m_1,m_2) & \text{if $k=1$, $\gamma = 1 \in \Gamma^{(1)}$, $m^{(1)} = (m_1,m_2) \in \mathfrak{M}^{(1,\gamma)}$}, \\ (\phi^{(k-1)}_{\gamma^{(k-1)}_1}(m^{(k-1)}_1), \phi^{(k-1)}_{\gamma^{(k-1)}_2})(m^{(k-1)}_2) & \text{if $k \ge 2$, $m^{(k)} \in \mathfrak{M}^{(k,\gamma)}$,} \\ & \text{ $\gamma = (\gamma^{(k-1)}_1, \gamma^{(k-1)}_2) \in \Gamma^{(k-1)} \times \Gamma^{(k-1)}$}, \\
& \text{$m^{(k)} = (m^{(k-1)}_1,m^{(k-1)}_2) \in \mathfrak{M}^{(k-1,\gamma^{(k-1)}_1)} \times \mathfrak{M}^{(k-1,\gamma^{(k-1)}_2)}$}.
\end{cases}
$$
We also define inductively the isomorphism $\chi^{(k)}_\gamma : \mathfrak{B}^{(k,\gamma)} \rightarrow \prod_{j=1}^{\mathfrak{d}(\gamma)} \mathbb{Z}$ by setting
$$
\chi^{(k)}_\gamma(\alpha^{(k)}) = \begin{cases} \alpha & \text{if $\gamma = 0\in \Gamma^{(k)}$ $\alpha^{(k)} = \alpha \in \mathbb{Z} = \mathfrak{B}^{(k,\gamma)}$}, \\ (\alpha_1, \alpha_2) & \text{if $k=1$, $\gamma = 1 \in \Gamma^{(1)}$, $\alpha^{(k)} = (\alpha_1, \alpha_2) \in \mathfrak{B}^{(1,\gamma)}$}, \\ (\chi^{(k-1)}_{\gamma^{(k-1)}_1} (\alpha^{(k-1)}_1), \chi^{(k-1)}_{\gamma^{(k-1)}_2} (\alpha^{(k-1)}_2)) & \text{if $\alpha^{(k)} \in \mathfrak{B}^{(k,\gamma)}$, \quad $k \ge 2$}, \\ & \text{ $\gamma = (\gamma^{(k-1)}_1,\gamma^{(k-1)}_2) \in \Gamma^{(k-1)} \times \Gamma^{(k-1)}$}, \\ & \text{$\alpha^{(k)} = (\alpha^{(k-1)}_1,\alpha^{(k-1)}_2) \in \mathfrak{B}^{(k-1,\gamma^{(k-1)}_1)} \times \mathfrak{B}^{(k-1,\gamma^{(k-1)}_2)}$}. \end{cases}
$$

These isomorphisms induce an ordering of the components of the corresponding vector. For that matter, given $1 \le i \le \mathfrak{d}(\gamma)$, we denote by $(m^{(k)})_i$ the $i$-th component of $\phi^{(k)}_\gamma(m^{(k)}) \in \prod_{j=1}^{\mathfrak{d}(\gamma)} \mathbb{Z}^\nu$. Note that for $m^{(k)} = (m^{(k-1)}_1,m^{(k-1)}_2) \in \mathfrak{M}^{(k-1,\gamma^{(k-1)}_1)} \times \mathfrak{M}^{(k-1,\gamma^{(k-1)}_2)}$, this ordering is lexicographical, that is, $(m^{(k)})_i = (m^{(k-1)}_1)_i$ if $1 \le i \le \mathfrak{d}(\gamma^{(k-1)}_1)$, $(m^{(k)})_{i+\mathfrak{d}(\gamma^{(k-1)}_1)} = (m^{(k-1)}_2)_i$ for $1 \le i \le \mathfrak{d}(\gamma^{(k-1)}_2)$.

$2$. This ordering helps us also to introduce the following sets,
\begin{equation}\label{eq:12Bsetsa}
\mathbb{I}^{(k,\gamma)} = \{ \alpha \in \mathbb{Z}^{\mathfrak{d}(\gamma)} : \sum_j \alpha_j = 1, \, \alpha_j \ge 0 \}
\end{equation}
and
\begin{equation}\label{eq:12Bsets}
\mathbb{A}^{(k,\gamma)} = \begin{cases} \{0 \in \mathbb{Z}\} & \text{if $\gamma = 0 \in \Gamma^{(k)}$}, \\ \{ (\alpha_1, \alpha_2) \in \mathbb{Z}^2 : \alpha_1 + \alpha_2 = 1, \, \alpha_j \ge 0 \} & \text{if $\gamma \in \Gamma^{(1)}$, $\gamma = 1$}, \\ \mathbb{A}^{(k-1,\gamma^{(k-1)}_1)} \times \mathbb{A}^{(k-1,\gamma^{(k-1)}_2)} + \mathbb{I}^{(k,\gamma)} & \text{if $\gamma \in \Gamma^{(k)}$, $k\ge 2$,} \\
& \text{ $\gamma = (\gamma^{(k-1)}_1, \gamma^{(k-1)}_2) \in \Gamma^{(k-1)} \times \Gamma^{(k-1)}$}. \end{cases}
\end{equation}
\end{defi}

Our next goal is to evaluate the functions $\mathfrak{P}(m^{(k)})$ in terms of the ``new variables'' $\alpha_j$.

\begin{lemma}\label{lem:Pexpansion}
Let $\gamma \in \Gamma^{(k)}$, $m^{(k)} \in \mathfrak{M}^{(k,\gamma)}$. Then,
\begin{equation}\label{eq:Pexpansionstatement}
\mathfrak{P}(m^{(k)}) \le \sum_{\alpha = (\alpha_{i})_{1 \le i \le \mathfrak{d}(\gamma)} \in \mathbb{A}^{(k,\gamma)}} \prod|(m^{(k)})_{i}|^{\alpha_{i}}.
\end{equation}
\end{lemma}

\begin{proof}
The statement obviously holds for $\gamma = 0 \in \Gamma^{(k)}$ since both sides in \eqref{eq:Pexpansionstatement} are equal to $1$ in this case. Let $k = 1$, $\gamma = 1$, $m^{(1)} = (m_1,m_2)$. Then the right-hand side in \eqref{eq:Pexpansionstatement} is equal $|m_1| + |m_2|$. By definition, $\mathfrak{P}(m^{(1)}) = |m_1 + m_2| \le |m_1|+|m_2|$. In particular, \eqref{eq:Pexpansionstatement} holds for $k = 1$.

Let $k \ge 2$. Assume that the statement holds for any $\gamma' \in \Gamma^{(k')}$ with $k' < k$. Let $\gamma = (\gamma^{(k-1)}_1, \gamma^{(k-1)}_2) \in \Gamma^{(k-1)} \times \Gamma^{(k-1)}$, $m^{(k)} \in \mathfrak{M}^{(k,\gamma)}$, $m^{(k)} = (m^{(k-1)}_1,m^{(k-1)}_2) \in \mathfrak{M}^{(k-1,\gamma^{(k-1)}_1)} \times \mathfrak{M}^{(k-1,\gamma^{(k-1)}_2)}$. Using this assumption, the inductive definition of $\mathfrak{P}(m^{(1)})$ in \eqref{eq:12FourierB2a} and
$\mathbb{A}^{(k,\gamma)}$ in \eqref{eq:12Bsets}, one obtains
\begin{align*}
\mathfrak{P}(m^{(k)}) & = \Big[\big| \sum_{j} \sum_{i=1}^{\mathfrak{d}(\gamma^{(k-1)}_{j})} (m^{(k-1)}_j)_{i} \big|\Big] \prod_j \mathfrak{P}(m^{(k-1)}_j) \\
& \le \Big[\sum_{j} \sum_{i=1}^{\mathfrak{d} (\gamma^{(k-1)}_j)} |(m^{(k-1)}_{j})_{i}|\Big] \prod_j \quad \quad \sum_{\alpha_j = (\alpha_{i,j})_{1 \le i \le \mathfrak{d}(\gamma^{(k-1)}_j)} \in \mathbb{A}^{(k-1,\gamma^{(k-1)}_j)}} \prod|(m^{(k-1)}_{j})_{i}|^{\alpha_{i,j}} \\
& = \sum_{\alpha = (\alpha_{i})_{1 \le i \le \mathfrak{d}(\gamma)} \in \mathbb{A}^{(k,\gamma)}} \prod|(m^{(k)})_{i}|^{\alpha_{i}},
\end{align*}
as claimed.
\end{proof}

We need the following elementary calculus lemma.

\begin{lemma}\label{lem:elementarycalculus}
$(1)$ For any $0 < \kappa \le 1$, we have
\[
\sum_{m\in\mathbb{Z}}\exp(-\kappa|m|)\le 6\kappa^{-1}.
\]

$(2)$ For every $\alpha = (\alpha_1,\dots,\alpha_r)$ with $\alpha_j \in \mathbb{Z}$, $\alpha_j \ge 0$, we have
$$
\sum_{m=(m_1,\dots,m_r):m_j \in \mathbb{Z}^{\nu}} \prod_j |m_j|^{\alpha_j} \exp(-\kappa |m_j|) \le \alpha! (C_0 \kappa^{-1})^{|\alpha|+r\nu },
$$
where $C_0$ is an absolute constant. Here,
\[
\alpha!:=\prod_j \alpha_j!,
\]
and $|n|$, as always, means the $\ell^1$-norm.

$(3)$ For every $\alpha = (\alpha_1,\dots,\alpha_r)$ with $\alpha_j \in \mathbb{Z}$, $\alpha_j \ge 0$ and $n\in \mathbb{Z}$, $n\ge 0$, we have
$$
\sum_{m=(m_1,\dots,m_r):m_j \in \mathbb{Z}^{\nu},\quad |m|\ge n} \prod_j |m_j|^{\alpha_j} \exp(-\kappa |m_j|) \le \alpha! (2C_0 \kappa^{-1})^{|\alpha|+r\nu } \exp (-\frac{\kappa}{2}n ).
$$

\end{lemma}
\begin{proof}
$(1)$ We have
\[
\sum_{m\in\mathbb{Z}}\exp(-\kappa|m|)=\frac{1+e^{-\kappa}}{1-e^{-\kappa}}\le \frac{2}{1-e^{-\kappa}}.
\]
Since $1-e^{-\kappa}\ge \frac{\kappa}{3}$ for $0<\kappa\le 1$, the estimate in $(1)$ holds.

$(2)$ Note that for any $\lambda\ge 0$ and any integer $a\ge 0$, we have
\[
\lambda^a=a!(2\kappa^{-1})^a\big[\frac{\Big(\frac{\kappa\lambda}{2}\Big)^a}{a!}\big]\le a!(2\kappa^{-1})^a\exp(\frac{\kappa\lambda}{2}).
\]
Therefore,
\begin{align*}
\sum_{m=(m_1,\dots,m_r):m_j \in \mathbb{Z}^{\nu}} \prod_j |m_j|^{\alpha_j} \exp(-\kappa |m_j|) & \le \sum_{m=(m_1,\dots,m_r):m_j \in \mathbb{Z}^{\nu}} \prod_j \alpha_j! (2\kappa^{-1})^{\alpha_j} \exp(-\frac{\kappa}{2} |m_j|) \\
& \le \alpha! (2\kappa^{-1})^{|\alpha|}\sum_{m=(m_1,\dots,m_r):m_j \in \mathbb{Z}^{\nu}} \prod_j  \exp(-\frac{\kappa}{2} |m_j|) \\
& = \alpha! (2\kappa^{-1})^{|\alpha|}\Big[\sum_{m \in \mathbb{Z}^{\nu}}\exp(-\frac{\kappa}{2} |m|)\Big]^r \\
& = \alpha! (2\kappa^{-1})^{|\alpha|}\Big[\sum_{m \in \mathbb{Z}}\exp(-\frac{\kappa}{2} |m|)\Big]^{r\nu}.
\end{align*}
Applying the estimate in $(1)$ one obtains the statement with $C_0=12$.

$(3)$ One has
\begin{align*}
\sum_{m=(m_1,\dots,m_r):m_j \in \mathbb{Z}^{\nu},\quad |m|=n} & \prod_j |m_j|^{\alpha_j} \exp(-\kappa |m_j|) \\
& = \sum_{m=(m_1,\dots,m_r):m_j \in \mathbb{Z}^{\nu},\quad |m|=n} \big[\prod_j\exp(-\frac{\kappa}{2} |m_j|)\big] \big[\prod_j |m_j|^{\alpha_j} \exp(-\frac{\kappa}{2} |m_j|)\big] \\
& \le \exp(-\frac{\kappa}{2}n)\sum_{m=(m_1,\dots,m_r):m_j \in \mathbb{Z}^{\nu}} \big[\prod_j |m_j|^{\alpha_j} \exp(-\frac{\kappa}{2} |m_j|)\big].
\end{align*}
Applying the estimate in $(2)$ one obtains the statement in $(3)$.
\end{proof}

Now we are ready to derive the central ingredients of our estimation method. Namely, in the next lemma, we first of all derive the crucial combinatorial identity in terms of the new variables $\alpha_j$ (which has nothing to do with the previous two lemmas). After that we obtain the first application of the identity by combining this it with the estimates from Lemmas~\ref{lem:Pexpansion} and \ref{lem:elementarycalculus} to evaluate the sums involving the functions
$\mathfrak{P}(m^{(k)})$.

\begin{lemma}\label{lem:PexpansionsumAA}
Let $\gamma \in \Gamma^{(k)}$.

$1$. For $\alpha \in \mathbb{A}^{(k,\gamma)}$, we have
$$
\sum_{i=1}^{\mathfrak{d}(\gamma)} \alpha_i = \mathfrak{l}(\gamma).
$$

\smallskip
$2$.
\begin{equation}\label{eq:Pexpansionsumtatement}
\sum_{m^{(k)} \in \mathfrak{M}^{(k,\gamma)}} \exp(-\kappa|m^{(k)}|) \mathfrak{P}(m^{(k)}) \le (C_0^\nu \kappa^{-\nu})^{(\mathfrak{l} (\gamma)+1)} \sum_{\alpha = (\alpha_{i})_{1 \le i \le \mathfrak{d} (\gamma)} \in \mathbb{Z}_+^{\mathfrak{d}(\gamma)} : |\alpha| = \mathfrak{l}(\gamma)} \prod \alpha_{i}! \; .
\end{equation}

\smallskip
$3$.
\begin{equation}\label{eq:PexpansionsumtatementN}
\sum_{m^{(k)} \in \mathfrak{M}^{(k,\gamma)} : \mu(m^{(k)}) = n} \exp(-\kappa |m^{(k)}|) \mathfrak{P}(m^{(k)}) \le (2^\nu C_0^\nu \kappa^{-\nu})^{(\mathfrak{l}(\gamma)+1)} \exp \Big( -\frac{\kappa |n|}{2} \Big) \sum_{\alpha = (\alpha_{i})_{1 \le i \le \mathfrak{d}(\gamma)} \in \mathbb{Z}_+^{\mathfrak{d}(\gamma)} : |\alpha| = \mathfrak{l}(\gamma)} \prod \alpha_{i} \; .
\end{equation}
\end{lemma}

\begin{proof}
$1$. The statement holds if $\gamma = 0 \in \Gamma^{(k)}$, since both sides in the identity are equal to $0$ in this case. For $k = 1$, $\gamma = 1$, the statement holds since both sides in the identity are equal to $1$ in this case. So, in particular, the statement holds for $k = 1$.

Let $k \ge 2$. Assume that the statement holds for any $\gamma' \in \Gamma^{(k')}$ with $k' < k$ and any $\alpha' \in \mathbb{A}^{(k',\gamma')}$. Let $\gamma = (\gamma^{(k-1)}_1,\gamma^{(k-1)}_2) \in \Gamma^{(k-1)} \times \Gamma^{(k-1)}$, $\alpha = (\alpha^{(1)}, \alpha^{(2)}) + \beta$, $\alpha^{(j)} \in \mathbb{A}^{(k-1,\gamma^{(k-1)}_j)}$, $\beta \in \mathbb{I}^{(k,\gamma)}$. Using the inductive assumption, one obtains
$$
\sum_{i=1}^{\mathfrak{d}(\gamma)} \alpha_i = \sum_{j=1,2} \sum_{i=1}^{\mathfrak{d} (\gamma^{(k-1)}_j)} \alpha^{(j)}_i + \sum_{i=1}^{\mathfrak{d}(\gamma)} \beta_i = \sum_{j=1,2} \mathfrak{l}(\gamma^{(k-1)}_j) + 1 = \mathfrak{l}(\gamma).
$$

$2$. This follows from Lemma~\ref{lem:Pexpansion} combined with Lemma~\ref{lem:elementarycalculus} and part $1$ of the current lemma.

$3$. Let $\mu(m^{(k)}) = n$. Then $|m^{(k)}|\ge |n|$. Note that for $\mu(m^{(k)}) = n$, one has $\exp(-\kappa |m^{(k)}|) \le \exp(-\frac{\kappa |n|}{2}) \exp(-\frac{\kappa |m^{(k)}|}{2})$. Applying part $2$ of the current lemma, one obtains the estimate.
\end{proof}

The next lemma demonstrates the effectiveness of the new variables $\alpha_j$ and the identity from the previous lemma.
Roughly speaking, it shows that the terms $ \frac{t^{\mathfrak{l}(\gamma)}}{\mathfrak{F}(\gamma)}$ from the estimate of the main function $I(t,m^{(k)})$ behave similarly to the terms of the series expansion for the multi-variable exponent $\exp [t^{\mathfrak{l}(\gamma)}]$. This gives a perfect estimate for the total sum of these terms, provided that $t$ is small enough.

\begin{lemma}\label{lem:PexpansionsumF}
For $0 < t \le 1/8$, we have
\begin{equation}\label{eq:PexpansionsumtatementPFst}
\sum_{\gamma \in \Gamma^{(k)}} \frac{t^{\mathfrak{l}(\gamma)}} {\mathfrak{F}(\gamma)} \sum_{\alpha = (\alpha_{i})_{1 \le i \le \mathfrak{d}(\gamma)}
\in \mathbb{A}^{(k,\gamma)}} \prod \alpha_{i}!\le 2.
\end{equation}
\end{lemma}

\begin{proof}
For $\gamma = 0 \in \Gamma^{(k)}$, we have $\mathfrak{l}(\gamma) = 0$, $\mathfrak{F}(\gamma) = 1$, $\mathbb{A}^{(k,\gamma)}= 0 \in \mathbb{Z}$, and hence
\begin{equation}\label{eq:PFstzerogamma}
\frac{t^{\mathfrak{l}(\gamma)}} {\mathfrak{F}(\gamma)} \sum_{\alpha = (\alpha_{i})_{1 \le i \le \mathfrak{d}(\gamma)} \in \mathbb{A}^{(k,\gamma)}} \prod \alpha_{i}! = 1.
\end{equation}

For $k = 1$, $\gamma = 1$, we have $\mathfrak{l}(\gamma) = 1$, $\mathfrak{F}(\gamma) = 1$, $\mathbb{A}^{(1,\gamma)} = \{(1,0),(0,1) \in \mathbb{Z}^2\}$, and hence
$$
\frac{t^{\mathfrak{l}(\gamma)}} {\mathfrak{F}(\gamma)} \sum_{\alpha = (\alpha_{i})_{1 \le i \le \mathfrak{d}(\gamma)} \in \mathbb{A}^{(k,\gamma)}} \prod \alpha_{i}! = 2t.
$$
In particular,
$$
\sum_{\gamma \in \Gamma^{(1)}} \frac{t^{\mathfrak{l}(\gamma)}} {\mathfrak{F}(\gamma)} \sum_{\alpha = (\alpha_{i})_{1 \le i \le \mathfrak{d}(\gamma)}
\in \mathbb{A}^{(k,\gamma)}} \prod \alpha_{i}! \le 1 + 2t < 2,
$$
and therefore \eqref{eq:PexpansionsumtatementPFst} holds when $k = 1$.

Let $k \ge 2$. Then,
\begin{align*}
\sum_{\gamma \in \Gamma^{(k)} \setminus \{0\}} & \frac{t^{\mathfrak{l}(\gamma)}} {\mathfrak{F}(\gamma)} \sum_{\alpha \in \mathbb{A}^{(k,\gamma)}} \prod_{i=1}^{\mathfrak{d}(\gamma)} \alpha_i! \\
& = \sum_{(\gamma^{(k-1)}_1, \gamma^{(k-1)}_2) \in \Gamma^{(k-1)} \times \Gamma^{(k-1)}} \frac{t}{\mathfrak{l}(\gamma^{(k-1)}_1) + \mathfrak{l}(\gamma^{(k-1)}_2)+1} \prod_{j=1,2} \frac{t^{\mathfrak{l}(\gamma^{(k-1)}_j)}} {\mathfrak{F}(\gamma^{(k-1)}_j)} \\
& \qquad \sum_{\alpha = (\alpha^{(1)}, \alpha^{(2)}) + \beta, \, \alpha^{(j)} \in \mathbb{A}^{(k-1,\gamma^{(k-1)}_j)}, \, \beta \in \mathbb{I}^{(k,(\gamma^{(k-1)}_1, \gamma^{(k-1)}_2))}} \prod_{i=1}^{\mathfrak{d}(\gamma)} \alpha_i! \\
& \le \sum_{(\gamma^{(k-1)}_1, \gamma^{(k-1)}_2) \in \Gamma^{(k-1)} \times \Gamma^{(k-1)}} \frac{t}{\mathfrak{l}(\gamma^{(k-1)}_1) + \mathfrak{l}(\gamma^{(k-1)}_2)+1} \sum_{\alpha^{(j)} \in \mathbb{A}^{(k-1,\gamma^{(k-1)}_j)}, \, j=1,2} \\
& \qquad \sum_{j_0=1,2} \sum_{i_0=1}^{\mathfrak{d}(\gamma^{(k-1)}_{j_0})} \prod_{j=1,2} \frac{t^{\mathfrak{l}(\gamma^{(k-1)}_j)}} {\mathfrak{F}(\gamma^{(k-1)}_j)} \prod_{i=1}^{\mathfrak{d}(\gamma^{(k-1)}_j)} ((\alpha^{(j)})_i + \delta_{i,i_0} \delta_{j,j_0})! \\
& = \sum_{(\gamma^{(k-1)}_1, \gamma^{(k-1)}_2) \in \Gamma^{(k-1)} \times \Gamma^{(k-1)}} \frac{t}{\mathfrak{l}(\gamma^{(k-1)}_1) + \mathfrak{l}(\gamma^{(k-1)}_2)+1} \sum_{\alpha^{(j)} \in \mathbb{A}^{(k-1,\gamma^{(k-1)}_j)}, \, j=1,2} \\
& \qquad \sum_{j_0=1,2} \sum_{i_0 = 1}^{\mathfrak{d} (\gamma^{(k-1)}_{j_0})} ((\alpha^{(j_0)})_{i_0} + 1) \prod_{j = 1,2} \frac{t^{\mathfrak{l} (\gamma^{(k-1)}_j)}} {\mathfrak{F} (\gamma^{(k-1)}_j)} \prod_{i=1}^{\mathfrak{d}(\gamma^{(k-1)}_j)} ((\alpha^{(j)})_i)! \\
& = \sum_{(\gamma^{(k-1)}_1, \gamma^{(k-1)}_2) \in \Gamma^{(k-1)} \times \Gamma^{(k-1)}} \frac{t}{\mathfrak{l}(\gamma^{(k-1)}_1) + \mathfrak{l}(\gamma^{(k-1)}_2)+1} \sum_{\alpha^{(j)} \in \mathbb{A}^{(k-1,\gamma^{(k-1)}_j)}, \, j=1,2} \\
& \qquad \sum_{j_0=1,2} (\mathfrak{l} (\gamma^{(k-1)}_{j_0}) + \mathfrak{d} (\gamma^{(k-1)}_{j_0})) \prod_{j=1,2} \frac{t^{\mathfrak{l} (\gamma^{(k-1)}_j)}} {\mathfrak{F}(\gamma^{(k-1)}_j)} \prod_{i=1}^{\mathfrak{d}(\gamma^{(k-1)}_j)} ((\alpha^{(j)})_i)! \\
& = 2 t \sum_{(\gamma^{(k-1)}_1, \gamma^{(k-1)}_2) \in \Gamma^{(k-1)} \times \Gamma^{(k-1)}} \sum_{\alpha^{(j)} \in \mathbb{A}^{(k-1,\gamma^{(k-1)}_j)}, \, j=1,2} \prod_{j=1,2} \frac{t^{\mathfrak{l} (\gamma^{(k-1)}_j)}} {\mathfrak{F}(\gamma^{(k-1)}_j)}
\prod_{i=1}^{\mathfrak{d}(\gamma^{(k-1)}_j)} ((\alpha^{(j)})_i)! \\
& = 2 t \prod_{j=1,2} \frac{t^{\mathfrak{l}(\gamma^{(k-1)}_j)}} {\mathfrak{F}(\gamma^{(k-1)}_j)} \sum_{\gamma^{(k-1)}_j \in \Gamma^{(k-1)}} \sum_{\alpha^{(j)} \in \mathbb{A}^{(k-1,\gamma^{(k-1)}_j)}} \prod_{i=1}^{\mathfrak{d}(\gamma^{(k-1)}_j)} ((\alpha^{(j)})_i)! \\
& \le 8t.
\end{align*}
Here we used the fact that for any $\gamma$, we have $\mathfrak{d}(\gamma) = \mathfrak{l}(\gamma)+1$, which implies the identity
\[
\sum_{j_0=1,2}\frac{\mathfrak{l}(\gamma^{(k-1)}_{j_0})+\mathfrak{d}(\gamma^{(k-1)}_{j_0})}{\mathfrak{l}(\gamma^{(k-1)}_1) + \mathfrak{l}(\gamma^{(k-1)}_2)+1}= 2
\]
for any $\gamma_1,\gamma_2$.

Combining this with \eqref{eq:PFstzerogamma}, we obtain \eqref{eq:PexpansionsumtatementPFst} when $k \ge 2$.
\end{proof}

\begin{corollary}\label{cor:12system1cestimatesNEW}
For $0 \le t \le \kappa^\nu/(8B_0C_0^\nu|\omega|)$, we have
$$
\sum_{\gamma \in \Gamma^{(k)}} \sum_{m^{(k)} \in \mathfrak{M}^{(k,\gamma)} : \mu(m^{(k)}) = n} B_0^{\mathfrak{d}(\gamma)} \exp(-\kappa |m^{(k)}|) \mathfrak{P}(m^{(k)}) \frac{(|\omega|t)^{\mathfrak{l}(\gamma)}}{\mathfrak{F}(\gamma)}\le 2B_0.
$$
\end{corollary}

\begin{proof}
The statement follows from part $2$ of Lemma~\ref{lem:PexpansionsumAA} combined with Lemma~\ref{lem:PexpansionsumF}. Here we use the identity
$\mathfrak{d}(\gamma)=\mathfrak{l}(\gamma)+1$. The factor $B_0$ in the right-hand side of the estimate appears due to the contribution
of the only term with $\gamma=0$.
\end{proof}

\begin{corollary}\label{cor:12system1cestimates} The following statements hold:

$(1)$ Set
$$
\mathfrak{C}(\mathfrak{m}) = \prod_j c(m_j), \quad \text{where $\mathfrak{m} = \{m_1, \dots, m_N\}$, $m_j \in \mathbb{Z}^\nu$}.
$$
Then, for $m^{(k)} \in \mathfrak{M}^{(k,\gamma)}$, we have
$$
|\mathfrak{C}(m^{(k)})| \le B_0^{\mathfrak{d}(\gamma)} \exp(-\kappa |m^{(k)}|).
$$

$(2)$ The functions $c_k(t,n)$ are well-defined and continuous for $0 \le t \le \kappa^\nu/(8B_0C_0^\nu|\omega|)$ and the following identities hold
$$
c_k(t,n) = \sum_{\gamma \in \Gamma^{(k)}} \sum_{m^{(k)} \in \mathfrak{M}^{(k,\gamma)} : \mu(m^{(k)}) = n} \mathfrak{C} (m^{(k)}) \mathfrak{f}^{(k,\gamma)} (m^{(k)}) I^{(k,\gamma)}(t,m^{(k)}).
$$
All the series involved converge absolutely and uniformly on the interval  $0 \le t \le \kappa^\nu/(8B_0C_0^\nu|\omega|)$.

$(3)$ For $0 < t <\kappa^\nu/(8B_02^\nu C_0^\nu|\omega|)$, we have
$$
|c_k(t,n)| \le 2 B_0 \exp \Big( -\frac{\kappa |n|}{2} \Big).
$$
\end{corollary}

\begin{proof}
$(1)$ The statement follows from the definition of $\mathfrak{M}^{(k,\gamma)}$ and $\mathfrak{C}(m^{(k)})$, and the decay assumption \eqref{eq:decay}.

$(2)$ The statement follows from the definitions \eqref{eq:c0def}--\eqref{eq:12FourierB2}. The absolute and uniform convergence of all the series involved is due to Corollary~\ref{cor:12system1cestimatesNEW}.

$(3)$ Due to part $2$, we have
$$
|c_k(t,n)| \le \sum_{\gamma \in \Gamma^{(k)}} \sum_{m^{(k)} \in \mathfrak{M}^{(k,\gamma)} : \mu(m^{(k)}) = n} B_0^{\mathfrak{d}(\gamma)} \exp(-\kappa |m^{(k)}|) \mathfrak{P}(m^{(k)}) \frac{(|\omega|t)^{\mathfrak{l}(\gamma)}}{\mathfrak{F}(\gamma)}.
$$
Combining this with part $3$ of Lemma~\ref{lem:PexpansionsumAA} and with Lemma~\ref{lem:PexpansionsumF}, one obtains the estimate.
\end{proof}

Let
\begin{align*}
\mathbb{I}^{(k)} & = \{ \alpha \in \mathbb{Z}^{k+1} : \sum_j \alpha_j=1, \quad \alpha_j \ge 0\}, \\
\mathbb{B}^{(k)} & = \begin{cases} \{ (\alpha_1, \alpha_2) \in \mathbb{Z}^2 : \alpha_1 + \alpha_2 = 1, \quad \alpha_j \ge 0\}, & k = 1, \\ \mathbb{B}^{(k-1)} \times \{ 0 \in \mathbb{Z} \} + \mathbb{I}^{(k)}, & k \ge 2. \end{cases}
\end{align*}
Notice that for any $\alpha \in \mathbb{B}^{(k)}$, we have
\begin{equation}\label{eq:12BsetsBBsumalphas}
\alpha \in \mathbb{R}^{k+1}, \quad \sum_j \alpha_j = k.
\end{equation}

\begin{lemma}\label{lem:12system1cestimates}
For $0 < t < \kappa^\nu/(8 B_0 2^\nu C_0^\nu |\omega|)$, we have
\begin{equation}\label{eq:12FourierB4eststsIkk}
|c_{k}(t,n) - c_{k-1}(t,n)| \le \frac{B_0^{k+1}(2 |\omega| t)^{k}}{k!} \sum_{m = (m_1, \dots, m_{k+1}) \in \mathbb{Z}^{(k+1)\nu} : \sum_j m_j = n} \sum_{\alpha \in \mathbb{B}^{(k)}} \prod_j |m_j|^{\alpha_j} \exp \Big( -\frac{\kappa}{2} |m_j| \Big).
\end{equation}
\end{lemma}

\begin{proof}
Recall that $c_0(t,n) := c(n) \exp(i t(n\omega)^3)$, and for $k = 1, 2, \dots$,
$$
c_k(t,n) = c(n) \exp(i t(n\omega)^3) -\frac{i n \omega}{2} \int_0^t \exp(i(t - \tau) (n\omega)^3) \sum_{m_1, m_2 \in \mathbb{Z}^\nu : m_1 + m_2 = n} c_{k-1}(\tau, m_1) c_{k-1}(\tau, m_2) \, d \tau,
$$
$n \in \mathbb{Z}^\nu$. In particular,
\begin{align*}
|c_1(t,n) - c_0(t,n)| & \le \frac{|n| |\omega|}{2} \int_0^t \sum_{m_1, m_2 \in \mathbb{Z}^\nu : m_1 + m_2 = n} |c_{0}(\tau,m_1)| |c_{0}(\tau,m_2)| \, d \tau \\
& \le \frac{B_0^2t|\omega|}{2} \sum_{m_1, m_2 \in \mathbb{Z}^\nu : m_1 + m_2 = n} \Big| \sum_j m_j \Big| \exp(-\kappa (|m_1| + |m_2|)).
\end{align*}
Thus, \eqref{eq:12FourierB4eststsIkk} holds for $k = 1$.

Let $k \ge 2$. Assume the estimate holds for any $1 \le k' \le k-1$. We have
\begin{align*}
|c_k(t,n) - c_{k-1}(t,n)| & \le \frac{|n| |\omega|}{2} \int_0^t \sum_{m_1, m_2 \in \mathbb{Z}^\nu : m_1 + m_2 = n} |c_{k-1}(\tau, m_1) c_{k-1}(\tau, m_2) - c_{k-2}(\tau, m_1) c_{k-2}(\tau, m_2)| \,  d \tau \\
& \le \frac{|n| |\omega|}{2} \int_0^t \sum_{m_1, m_2 \in \mathbb{Z}^\nu : m_1 + m_2 = n} |c_{k-1}(\tau, m_1) - c_{k-2}(\tau, m_1)| |c_{k-1}(\tau, m_2)| \, d \tau \\
& \qquad + \frac{|n| |\omega|}{2} \int_0^t \sum_{m_1, m_2 \in \mathbb{Z}^\nu : m_1 + m_2 = n} |c_{k-1}(\tau, m_2) - c_{k-2}(\tau, m_2)| |c_{k-2}(\tau, m_1)| \, d \tau.
\end{align*}
Using the inductive assumption and the estimate from Corollary~\ref{cor:12system1cestimates}, we obtain
\begin{align*}
\frac{|n| |\omega|}{2} & \int_0^t \sum_{n_1, n_2 \in \mathbb{Z}^\nu : n_1 + n_2 = n} |c_{k-1}(\tau, n_1) - c_{k-2}(\tau, n_1)| |c_{k-1}(\tau, n_2)|
\, d \tau \\
& \le \frac{|\omega|}{2} \int_0^t \sum_{n_1, n_2 \in \mathbb{Z}^\nu : n_1 + n_2 = n} \frac{B_0^{k}(2 |\omega| \tau)^{k-1}}{(k-1)!} \sum_{m = (m_1, \dots, m_k) \in \mathbb{Z}^{k\nu} : \sum_j m_j = n_1} \Big| \Big( \sum_j m_j \Big) + n_2 \Big| \times \\
& \qquad \sum_{\alpha \in \mathbb{B}^{(k-1)}} \prod_j |m_j|^{\alpha_j} \exp \Big( -\frac{\kappa}{2} |m_j| \Big) \Big( 2 B_0 \exp \Big( -\frac{\kappa}{2} |n_2| \Big) \Big) \, d \tau \\
& \le \frac{2^{k-1} B_0^{k+1} (|\omega| t)^{k}}{k!} \sum_{m = (m_1, \dots, m_{k+1}) \in \mathbb{Z}^{(k+1) \nu} : \sum_j m_j = n} \sum_{\alpha \in \mathbb{B}^{(k)}} \prod_j |m_j|^{\alpha_j} \exp \Big( -\frac{\kappa}{2} |m_j| \Big).
\end{align*}
Similarly,
\begin{align*}
\frac{|n| |\omega|}{2} & \int_0^t \sum_{n_1, n_2 \in \mathbb{Z}^\nu : n_1 + n_2 = n} |c_{k-1}(\tau, n_2) - c_{k-2}(\tau, n_2)| |c_{k-2}(\tau, n_1)|
\, d \tau \\
& \le \frac{2^{k-1} B_0^{k+1} (|\omega| t)^{k}}{k!} \sum_{m = (m_1, \dots, m_{k+1}) \in \mathbb{Z}^{(k+1) \nu} : \sum_j m_j = n} \sum_{\alpha \in \mathbb{B}^{(k)}} \prod_j |m_j|^{\alpha_j} \exp \Big( -\frac{\kappa}{2} |m_j| \Big).
\end{align*}
Putting the three estimates together, the assertion follows.
\end{proof}

\begin{corollary}\label{cor:12system1cestimatesBk}
With the constant $C_0$ from Lemma~\ref{lem:elementarycalculus}, we have for $0 < t < \kappa^\nu/(8 B_0 2^\nu C_0^\nu |\omega|)$,
\begin{equation}\label{eq:B4eststsIkkfactorial}
|c_{k}(t,n) - c_{k-1}(t,n)| \le \frac{B_0^{k+1} (4C_0 \kappa^{-1} |\omega| t)^{k}}{k!} \exp \Big( -\frac{\kappa}{4} |n| \Big) \sum_{\alpha \in \mathbb{B}^{(k)}} \prod_j \alpha_j! \; .
\end{equation}
\end{corollary}

\begin{proof}
Due to Lemma~\ref{lem:12system1cestimates}, we have
\begin{align*}
|c_{k}(t,n) & - c_{k-1}(t,n)| \le \frac{B_0^{k+1}(2 |\omega|t)^{k}}{k!} \sum_{m = (m_1, \dots, m_{k+1}) \in \mathbb{Z}^{(k+1)\nu} : \sum_j m_j = n} \sum_{\alpha \in \mathbb{B}^{(k)}} \prod_j |m_j|^{\alpha_j} \exp \Big( -\frac{\kappa}{2} |m_j| \Big) \\
& \le \frac{B_0^{k+1}(2 |\omega|t)^{k}}{k!} \exp \Big( -\frac{\kappa}{4} |n| \Big) \sum_{m = (m_1, \dots, m_{k+1}) \in \mathbb{Z}^{(k+1)\nu} : \sum_j m_j = n} \sum_{\alpha \in \mathbb{B}^{(k)}} \prod_j |m_j|^{\alpha_j} \exp \Big( -\frac{\kappa}{4} |m_j| \Big)\\
& \le \frac{B_0^{k+1}(2 |\omega|t)^{k}}{k!} \exp \Big( -\frac{\kappa}{4} |n| \Big) \sum_{m = (m_1, \dots, m_{k+1}) \in \mathbb{Z}^{(k+1)\nu} : |m| \ge n} \sum_{\alpha \in \mathbb{B}^{(k)}} \prod_j |m_j|^{\alpha_j} \exp \Big( -\frac{\kappa}{4} |m_j| \Big).
\end{align*}
Combining this estimate with $(3)$ in Lemma~\ref{lem:elementarycalculus} and with \eqref{eq:12BsetsBBsumalphas}, we obtain the statement.
\end{proof}

\begin{remark}\label{rem.cont345}
We now need to estimate the sum on the right-hand side of \eqref{eq:B4eststsIkkfactorial}. For a combinatorial argument related to this task, we need to introduce the following mappings. Let $N$, $\ell$ be arbitrary. Set $\mathfrak{A}_N(\ell) = \{\alpha = (\alpha_1, \dots, \alpha_N) \in \mathbb{Z}^N : \alpha_j \ge 0, \quad \sum_j \alpha_j = \ell \}$. Given $\alpha = (\alpha_1, \dots, \alpha_N) \in \mathbb{Z}^N$ with $\alpha_j \ge 0$, set $\mathcal{S}(\alpha) = \{ i : \alpha_i > 0 \}$, $\mathcal{E}(\alpha) = \{ j : \alpha_j = \min_{\mathcal{S}(\alpha)} \alpha_i \}$. We enumerate $\mathcal{E}(\alpha)$ in increasing order: $\mathcal{E}(\alpha)  = \{ j_1(\alpha) < \cdots \}$. Furthermore, set $\Phi(\alpha) = (\phi_1(\alpha), \dots, \phi_N(\alpha))$, where $\phi_j(\alpha) = \alpha_j$ if $j \neq j_1(\alpha)$, $\phi_{j_1(\alpha)}(\alpha) = \alpha_{j_1(\alpha)} - 1$.
\end{remark}

\begin{lemma}\label{lem:alphamapred}
$0$. $\Phi$ maps $\mathfrak{A}_N(\ell)$ into $\mathfrak{A}_N(\ell-1)$.

\smallskip
$1$. $\phi_{j_1(\alpha)}(\alpha) < \min_{j : \phi_j(\alpha) > 0, \; j \neq j_1(\alpha)} \phi_j(\alpha)$. Here, by convention,
minimum over an empty set is set to be $+\infty$.

\smallskip
$2$. If $\Phi(\alpha) = \Phi(\alpha')$, then $\alpha_j = \alpha'_j$ for $j \notin \{j_1(\alpha), j_1(\alpha')\}$.

\smallskip
$3$. If $\Phi(\alpha) = \Phi(\alpha')$ and $j_1(\alpha) = j_1(\alpha')$, then $\alpha = \alpha'$.

\smallskip
$4$. For any $\beta$, $\mathrm{card}( \Phi^{-1}(\beta) ) \le N$.

\smallskip
$5$. If $\Phi(\alpha) = \Phi(\alpha')$ and $\alpha_{j_1(\alpha)} > 1$, $\alpha'_{j_1(\alpha')} > 1$, then $\alpha = \alpha'$.
\end{lemma}

\begin{proof}
$0$. This follows from the definition of the map $\Phi$.

$1$. We have
\begin{align*}
\phi_{j_1(\alpha)}(\alpha) & = \alpha_{j_1(\alpha)} - 1 = \Big( \min_{j : \alpha_j > 0} \alpha_j \Big) - 1 \le \Big( \min_{j: \alpha_j > 0, \, j \neq j_1(\alpha)} \alpha_j \Big) - 1 \\
& = \min_{j : \phi_j(\alpha) > 0, \, j \neq j_1(\alpha)} \phi_j(\alpha) - 1 < \min_{j : \phi_j(\alpha) > 0, \, j \neq j_1(\alpha)} \phi_j(\alpha).
\end{align*}

$2$. This follows from the definition of $\Phi$.

$3$. Assume $\Phi(\alpha) = \Phi(\alpha')$ and $j_1(\alpha) = j_1(\alpha')$. Due to part $(2)$ of the current lemma, we have $\alpha_j = \alpha'_j$ for $j \notin \{j_1(\alpha)\}$. Furthermore, $\alpha_{j_1(\alpha)} = \phi_{j_1(\alpha)}(\alpha) + 1 = \phi_{j_1(\alpha')}(\alpha') + 1 = \alpha_{j_1(\alpha')}$. Thus, $\alpha = \alpha'$.

$4$. This follows from part $3$ of the current lemma.

$5$. Assume $\Phi(\alpha) = \Phi(\alpha') =: \beta$ and $\alpha_{j_1(\alpha)} > 1$, $\alpha'_{j_1(\alpha')} > 1$. Assume that $j_1(\alpha) \neq j_1(\alpha')$. Then due to part $1$ of the current lemma, we have $\beta_{j_1(\alpha)} < \min_{j : \beta_j > 0, \, j \neq j_1(\alpha)} \beta_j$. Note that $\beta_{j_1(\alpha')} = \alpha'_{j_1(\alpha')} - 1 > 0$. Since we assume $j_1(\alpha) \neq j_1(\alpha')$, we conclude that $\beta_{j_1(\alpha)} < \beta_{j_1(\alpha')}$. Similarly, $\beta_{j_1(\alpha')} < \beta_{j_1(\alpha)}$, which is obviously impossible. Thus,  $j_1(\alpha) = j_1(\alpha')$. Now the statement follows from part $3$ of the current lemma.
\end{proof}

Now everything is ready for us to finalize our estimation method.

\begin{lemma}\label{lem:PexpansionsumestA}
$1$. For any $\ell \le N$, we have
$$
\sum_{\alpha = (\alpha_1, \dots, \alpha_N) \in \mathfrak{A}_N(\ell)} \prod_i \alpha_{i}! \le (\ell + N) \sum_{\alpha = (\alpha_1, \dots, \alpha_N) \in \mathfrak{A}_N(\ell-1)} \prod_i \alpha_{i}! \; .
$$

$2$. We have
$$
\sum_{\alpha = (\alpha_1, \dots, \alpha_N) \in \mathfrak{A}_N(\ell)} \prod_i \alpha_{i}! < (2N)^\ell.
$$
\end{lemma}

\begin{proof}
$1$. Note that
\begin{align}\label{eq:Pexpansionsumtatementest1aNNN}
\sum_{\alpha = (\alpha_1, \dots, \alpha_N) \in \mathfrak{A}_N(\ell)} \prod_i \alpha_{i}! & = \sum_{\alpha = (\alpha_1, \dots, \alpha_N) \in \mathfrak{A}_N(\ell)} \alpha_{j_1(\alpha)} (\alpha_{j_1(\alpha)} - 1)! \prod_{i \neq j_1(\alpha)} \alpha_{i}! \\
& = \sum_{\alpha = (\alpha_1, \dots, \alpha_N) \in \mathfrak{A}_N(\ell)} \alpha_{j_1(\alpha)} \prod_{i} \phi_i(\alpha)! \; .
\end{align}

Recall that due to Lemma~\ref{lem:alphamapred}, $\Phi$ maps $\mathfrak{A}_N(\ell)$ into $\mathfrak{A}_N(\ell-1)$. Let $\mathfrak{A}'_N(\ell) = \{ \alpha \in \mathfrak{A}_N(\ell) : \alpha_{j_1(\alpha)} > 1 \}$, $\mathfrak{A}''_N(\ell) = \mathfrak{A}_N(\ell) \setminus \mathfrak{A}'_N(\ell)$. Recall also that due to Lemma~\ref{lem:alphamapred}, $\Phi$ is injective on $\mathfrak{A}'_N(\ell)$ and $\mathrm{card}( \Phi^{-1}(\beta) ) \le N$ for any $\beta$. Hence, due to the identity above, we have 
\begin{align*}
\sum_{\alpha = (\alpha_1, \dots, \alpha_N) \in \mathfrak{A}_N(\ell)} \prod_i \alpha_{i}! & = \sum_{\alpha = (\alpha_1, \dots, \alpha_N) \in \mathfrak{A}_N(\ell)} \alpha_{j_1(\alpha)} \prod_{i} \phi_i(\alpha)! \\
& = \sum_{\alpha \in \mathfrak{A}'_N(\ell)} \alpha_{j_1(\alpha)} \prod_{i} \phi_i(\alpha)! + \sum_{\alpha \in \mathfrak{A}''_N(\ell)} \alpha_{j_1(\alpha)} \prod_{i} \phi_i(\alpha)! \\
& = \sum_{\alpha \in \mathfrak{A}'_N(\ell)} \alpha_{j_1(\alpha)} \prod_{i} \phi_i(\alpha)! + \sum_{\alpha \in \mathfrak{A}''_N(\ell)} \prod_{i} \phi_i(\alpha)! \\
& \le \ell \sum_{\beta \in \mathfrak{A}_N(\ell-1)} \prod_{i} \beta_i! + N \sum_{\beta \in \mathfrak{A}_N(\ell-1)} \prod_{i} \beta_i! \; ,
\end{align*}
as claimed.

$2$. This follows from part $1$.
\end{proof}

\begin{corollary}\label{cor:cestimatesBkRE}
With the constant $C_0$ from Lemma~\ref{lem:elementarycalculus}, we have for $0 < t < \kappa^\nu/(8 B_0 2^\nu C_0^\nu |\omega|)$,
$$
|c_{k}(t,n) - c_{k-1}(t,n)| \le B_0^{k+1}(4^{\nu+1} C_1 \kappa^{-\nu} |\omega| t)^{k} \exp \Big( -\frac{\kappa}{4} |n| \Big),
$$
where $C_1 > C_0$ is an absolute constant.
\end{corollary}

\begin{proof}
Due to Corollary~\ref{cor:12system1cestimatesBk} we have for $0 < t < \kappa^\nu/(8 B_0 2^\nu C_0^\nu |\omega|)$,
\begin{equation}\label{eq:B4eststsIkkfactorialInv1}
|c_{k}(t,n) - c_{k-1}(t,n)| \le \frac{B_0^{k+1} (4^{\nu+1} C_0^\nu \kappa^{-\nu} |\omega| t)^{k}}{k!} \exp \Big( -\frac{\kappa}{4} |n| \Big) \sum_{\alpha \in \mathbb{B}^{(k)}} \prod_j \alpha_j! \; .
\end{equation}
Recall that for any $\alpha \in \mathbb{B}^{(k)}$ holds
\begin{equation}\nn
\alpha\in \mathbb{R}^{k+1}, \quad \sum_j \alpha_j = k.
\end{equation}
Therefore, due to part $2$ of Lemma~\ref{lem:PexpansionsumestA} with $N=k+1$, $\ell=k$ we have also
\begin{equation}\label{summationbasic}
\sum_{\alpha \in \mathbb{B}^{(k)}} \prod_j \alpha_j!\le \sum_{\alpha = (\alpha_1, \dots, \alpha_N) \in \mathfrak{A}_N(N)} \prod_i \alpha_{i}! < (2N)^k.
\end{equation}
Due to Stirling's formula,
\[
k! \gtrsim k^k e^{-k},\quad (k!)^{-1}(2N)^k\lesssim (2e)^{k},
\]
and the statement follows.
\end{proof}

\section{Proof of the Main Results}\label{sec.3}

We first prove Theorem~A. We derive the existence of the solution in Theorem~A in a straightforward way from Corollary~\ref{cor:cestimatesBkRE}, combined with Lemma~\ref{lem:12Fourier2}. The proof of the uniqueness requires some extra work. We start with some auxiliary statements needed in the proof of the uniqueness of the solution.

The following lemma and its proof are well known.

\begin{lemma}\label{lem:12Fourier1U}
Let $\alpha_n \in \mathbb{R}$, $h(n) \in \mathbb{C}$, $n \in \mathbb{Z}^\nu$. Assume that
\begin{itemize}

\item[(a)] $\alpha_m \neq \alpha_n$, unless $m = n$,

\item[(b)] $\sum_{n  \in \mathbb{Z}^\nu} |h(n)| < \infty$.

\end{itemize}
If
$$
\sum_n h(n) \, e^{i \alpha_n x} = 0
$$
for all $x \in \mathbb{R}$, then $h(n) = 0$ for all $n \in \mathbb{Z}^\nu$.
\end{lemma}

\begin{proof}
One has
$$
\lim_{A \rightarrow \infty} \frac{1}{2A} \int_{-A}^{A} e^{i \beta x} \, dx = \begin{cases} 0 \quad \text{if $\beta \neq 0$}, \\
1 \quad \text{if $\beta = 0$}.
\end{cases}
$$
Therefore, for each $m \in \mathbb{Z}^\nu$,
$$
0 = \lim_{A \rightarrow \infty} \frac{1}{2A} \int_{-A}^{A} \sum_{n  \in \mathbb{Z}^\nu} h(n) e^{i \alpha_n x} e^{-i \alpha_m x} \, dx = \sum_{n \in \mathbb{Z}^\nu} h(n)
\lim_{A \rightarrow \infty} \frac{1}{2A} \int_{-A}^{A} e^{i (\alpha_n - \alpha_m) x} \, dx = h(m).
$$
The correctness of the calculation here is due to conditions $(a)$ and $(b)$.
\end{proof}

The lemma we just proved helps us to show that under very natural conditions, the KdV equation implies a system of integral equations for the Fourier coefficients of the solution.

\begin{lemma}\label{lem:12Fourier2U}
Assume that $v$ obeys the KdV equation
\begin{equation}\label{eq:KdVU1}
\partial_t v + \partial^3_x v + v\partial_x v = 0.
\end{equation}
Assume also that the following expansion holds,
$$
v(t,x) = \sum_{n \in \mathbb{Z}^\nu} h(t,n) \exp(i x n \omega),
$$
where $\omega \in \mathbb{R}^\nu$ is such that $\omega n\neq 0$ for every $n \not= 0$, and the Fourier coefficients $h(t,n)$ obey
$$
\sum_{n \in \mathbb{Z}^\nu} |h(t,n)| |n|^3 < \infty.
$$
Then, the following integral equations hold:
$$
h(t,n) = h(0,n) e^{i t (n\omega)^3} -\sum_m \int_0^t h(\tau,n-m) h(\tau,m) (i m \omega) e^{i(t-\tau) (n\omega)^3} \, d\tau, \quad n \in \mathbb{Z}^\nu.
$$
\end{lemma}

\begin{proof}
We have
$$
\partial^\alpha_x v = \sum_n h(t,n) (i n \omega)^\alpha e^{i x n \omega}, \quad \alpha \le 3,
$$
$$
v \partial_x v = \sum_n \big[ \sum_m h(t,n-m) h(t,m) (i m \omega) \big] e^{i x n \omega},
$$
$$
\sum_n |h(t,n) (in\omega)^3| < \infty,
$$
$$
\sum_n \big| \sum_m h(t,n-m) h(t,m) (i m \omega) \big| < \infty.
$$
Using equation \eqref{eq:KdVU1}, we obtain
\begin{align*}
\sum_n h(t,n) e^{i x n \omega} & = v(t,x) \\
& = v(0,x) - \int_0^t [\partial^3_x v(x,\tau) + v(x,\tau) \partial_x v(x,\tau)] \, d\tau \\
& = \sum_n h(0,n) e^{i x n \omega} - \sum_n \Big[ \int_0^t h(\tau,n) (i n \omega)^3 \, d\tau \Big] e^{i x n \omega} \\
& \qquad - \sum_n \Big[ \sum_m \int_0^t h(\tau,n-m) h(\tau,m) (i m \omega) \, d\tau \Big] e^{i x n \omega}.
\end{align*}
Due to Lemma~\ref{lem:12Fourier1U} and the assumption $\omega n\neq 0$ for $n \not= 0$, this implies
\begin{equation}\label{eq:Fourier2UP2}
h(t,n) = h(0,n) - \int_0^t h(\tau,n) (i n \omega)^3 \, d\tau -\sum_m \int_0^t h(\tau,n-m) h(\tau,m) (i m \omega) \, d\tau, \quad n \in \mathbb{Z}^\nu.
\end{equation}
It follows from \eqref{eq:Fourier2UP2} that $\partial_t h(t,n)$ exist and obey
\begin{equation}\label{eq:Fourier2UP3}
\partial_\tau h(\tau,n) = -h(\tau,n) (i n \omega)^3 - \sum_m h(\tau,n-m) h(\tau,m) (i m \omega), \quad n \in \mathbb{Z}^\nu.
\end{equation}
Multiplying both sides of \eqref{eq:Fourier2UP3} by $e^{i(t-\tau) (n\omega)^3}$ and integrating from $\tau = 0$ to $\tau = t$, we obtain
\begin{equation}\label{eq:Fourier2UP8}
\begin{split}
\int_0^t \partial_\tau h(\tau,n) e^{i(t-\tau) (n\omega)^3} \, d\tau = -\int_0^t h(\tau,n) (i n \omega)^3 e^{i (t-\tau) (n\omega)^3} \, d\tau \\
-\sum_m \int_0^t h(\tau,n-m) h(\tau,m) (i m \omega) e^{i(t-\tau) (n\omega)^3} \, d\tau, \quad n \in \mathbb{Z}^\nu.
\end{split}
\end{equation}
Integration by parts on the left-hand side gives
\begin{equation}\label{eq:Fourier2UP9}
\int_0^t \partial_\tau h(\tau,n) e^{i(t-\tau) (n\omega)^3} \, d\tau = h(\tau,n) e^{i(t-\tau)(n\omega)^3} \Big|_{\tau = 0}^{\tau = t} + \int_0^t h(\tau,n) (i (n \omega)^3) e^{i (t-\tau) (n\omega)^3} \, d\tau.
\end{equation}
Combining \eqref{eq:Fourier2UP8} with \eqref{eq:Fourier2UP9} (and noting that $(i n \omega)^3 = - i (n \omega)^3$), we obtain
$$
h(t,n) - h(0,n) e^{i t (n\omega)^3} = -\sum_m \int_0^t h(\tau,n-m) h(\tau,m) (i m \omega) e^{i (t-\tau) (n\omega)^3} \, d\tau, \quad n \in \mathbb{Z}^\nu,
$$
as claimed.
\end{proof}

Now we employ the integral equations to compare the Fourier coefficients of two solutions of the KdV equation which originate
from the same quasi-periodic initial data. To do the estimations we need an a priori exponential decay estimate for the decay
of the coefficients. This is because we invoke the estimation of the sums in terms of the ``new variables'' $\alpha_j$ from Section~\ref{sec.2}, which require the exponential decay. Recall that the exponential decay condition is required in the uniqueness statement in Theorem~A.

\begin{lemma}\label{lem:12Fouriercontract}
Let $c(t,n)$, $h(t,n)$ be functions of $t \in [0,t_0)$, $t_0 > 0$, $n \in \mathbb{Z}^\nu$, which obey $|c(t,n)|, |h(t,n)| \le B \exp(-\rho |n|)$,
$n \in \mathbb{Z}^\nu$, $B,\rho > 0$. Assume that the following equations hold:
\begin{equation}\label{eq:Fouriercontrac1}
\begin{split}
c(t,n) = c(0,n) e^{it(n\omega)^3} - \sum_m \int_0^t c(\tau,n-m) c(\tau,m) (i m \omega) e^{i(t-\tau) (n\omega)^3} \, d \tau, \quad n \in \mathbb{Z}^\nu, \\
h(t,n) = h(0,n) e^{it(n\omega)^3} - \sum_m \int_0^t h(\tau,n-m) h(\tau,m) (i m \omega) e^{i(t-\tau) (n\omega)^3} \, d \tau, \quad n \in \mathbb{Z}^\nu.
\end{split}
\end{equation}
Assume also that $h(0,n) = c(0,n)$ for all $n \in \mathbb{Z}^\nu$. Then, for $k = 1,2,\dots$, we have
\begin{equation}\label{eq:Fouriercontrac2ST}
|h(t,n) - c(t,n)| \le \frac{B^{k+1}(|\omega| t)^{k}}{k!} \sum_{m = (m_1, \dots, m_{k+1}) \in \mathbb{Z}^{(k+1)\nu} : \sum_j m_j = n} \sum_{\alpha \in \mathbb{B}^{(k)}} \prod_j |m_j|^{\alpha_j} \exp \Big( -\rho |m_j| \Big),
\end{equation}
where $\mathbb{B}^{(k)}$ is defined as in Lemma~\ref{lem:12system1cestimates}.
\end{lemma}

\begin{proof}
It is convenient to rewrite \eqref{eq:Fouriercontrac1} as follows,
\begin{equation}\label{eq:Fouriercontrac1a}
\begin{split}
c(t,n) = c(0,n) e^{it(n\omega)^3} - \frac{i n \omega}{2} \sum_{m_1, m_2 \in \mathbb{Z}^\nu : m_1 + m_2 = n} \int_0^t c(\tau,m_1) c(\tau,m_2) e^{i(t-\tau) (n\omega)^3} \, d \tau,\quad n \in \mathbb{Z}^\nu, \\
h(t,n) = h(0,n) e^{it(n\omega)^3} - \frac{i n \omega}{2} \sum_{m_1, m_2 \in \mathbb{Z}^\nu : m_1 + m_2 = n} \int_0^t h(\tau,m_1) h(\tau,m_2) e^{i(t-\tau) (n\omega)^3} \, d \tau, \quad n \in \mathbb{Z}^\nu.
\end{split}
\end{equation}
Subtracting in \eqref{eq:Fouriercontrac1a}, we obtain
\begin{align*}
|h(t,n) - c(t,n)| & \le \frac{|n||\omega|}{2} \sum_{m_1, m_2 \in \mathbb{Z}^\nu : m_1 + m_2 = n} \int_0^t |h(\tau,m_1) h(\tau,m_2) - c(\tau,m_1) c(\tau,m_2)| \, d \tau \\
& \le B^2 t |\omega| \sum_{m_1, m_2 \in \mathbb{Z}^\nu : m_1 + m_2 = n} \Big| \sum_j m_j \Big| \exp(-\rho(|m_1| + |m_2|)).
\end{align*}
So, \eqref{eq:Fouriercontrac2ST} holds for $k=1$.

Assume \eqref{eq:Fouriercontrac2ST} holds for $k-1$. We have
\begin{align*}
|h & (t,n) - c(t,n)| \le \frac{|n| |\omega|}{2} \sum_{m_1, m_2 \in \mathbb{Z}^\nu : m_1 + m_2 = n} \int_0^t |h(\tau,m_1) h(\tau,m_2) - c(\tau,m_1) c(\tau,m_2)| \, d \tau \\
& \le \frac{|n| |\omega|}{2} \sum_{m_1, m_2 \in \mathbb{Z}^\nu : m_1 + m_2 = n} \int_0^t [|h(\tau,m_1) - c(\tau,m_1)| |h(\tau,m_2)| + |h(\tau,m_2) - c(\tau,m_2)| |c(\tau,m_1)|] \, d \tau.
\end{align*}
Using the inductive assumption, we obtain
\begin{align*}
\frac{|n| |\omega|}{2} & \sum_{n_1, n_2 \in \mathbb{Z}^\nu : n_1 + n_2 = n} \int_0^t |h(\tau,n_1) - c(\tau,n_1)| |h(\tau,n_2)| \, d \tau \\
& \le \frac{|\omega|}{2} \int_0^t \sum_{n_1, n_2 \in \mathbb{Z}^\nu : n_1 + n_2 = n} \frac{B^{k}( |\omega| \tau)^{k-1}}{(k-1)!} \sum_{m = (m_1, \dots, m_k) \in \mathbb{Z}^{k\nu} : \sum_j m_j = n_1} \Big| \Big( \sum_j m_j \Big) + n_2 \Big| \times \\
& \qquad \sum_{\alpha \in \mathbb{B}^{(k-1)}} \prod_j |m_j|^{\alpha_j} \exp \Big( -\rho |m_j| \Big) \Big( B \exp \Big( -\rho|n_2| \Big) \Big) \, d \tau \\
& \le \frac{ B^{k+1} (|\omega| t)^{k}}{2k!} \sum_{m = (m_1, \dots, m_{k+1}) \in \mathbb{Z}^{(k+1) \nu} : \sum_j m_j = n} \sum_{\alpha \in \mathbb{B}^{(k)}} \prod_j |m_j|^{\alpha_j} \exp \Big( -\rho |m_j| \Big).
\end{align*}
Similarly,
\begin{align*}
\frac{|n| |\omega|}{2} & \sum_{n_1, n_2 \in \mathbb{Z}^\nu : n_1 + n_2 = n} \int_0^t |h(\tau,n_2) - c(\tau,n_2)| |c(\tau,n_2)| \, d \tau \\
& \le \frac{ B^{k+1} (|\omega| t)^{k}}{2k!} \sum_{m = (m_1, \dots, m_{k+1}) \in \mathbb{Z}^{(k+1) \nu} : \sum_j m_j = n} \sum_{\alpha \in \mathbb{B}^{(k)}} \prod_j |m_j|^{\alpha_j} \exp \Big( -\rho |m_j| \Big),
\end{align*}
and the assertion follows.
\end{proof}

\begin{corollary}\label{cor:12Fouriercontract}
Let $h(t,n), c(t,n)$ be as in Lemma~\ref{lem:12Fouriercontract}. With the constant $C_0$ from Lemma~\ref{lem:elementarycalculus}, we have for $k=1,\dots$,
\begin{equation}\label{eq:B4eststsIkkfactorialU1}
|h(t,n) - c(t,n)| \le \frac{B^{k+1} (4^{\nu+1}C_0^\nu \rho^{-(\nu+1)} |\omega| t)^{k}}{k!} \sum_{\alpha \in \mathbb{B}^{(k)}} \prod_j \alpha_j! \; .
\end{equation}
\end{corollary}

\begin{proof}
Due to Lemma~\ref{lem:12Fouriercontract}, we have
\begin{align*}
|h(t,n) & - c(t,n)| \le \frac{B^{k+1}( |\omega|t)^{k}}{k!} \sum_{m = (m_1, \dots, m_{k+1}) \in \mathbb{Z}^{(k+1)\nu} : \sum_j m_j = n} \sum_{\alpha \in \mathbb{B}^{(k)}} \prod_j |m_j|^{\alpha_j} \exp \Big( -\rho |m_j| \Big) \\
& = \frac{B^{k+1}( |\omega|t)^{k}}{k!} \sum_{m = (m_1, \dots, m_{k+1}) \in \mathbb{Z}^{(k+1)\nu} : \sum_j m_j = n}
\prod_j \exp \Big( -\frac{\rho}{2}|m_j|\Big)
 \sum_{\alpha \in \mathbb{B}^{(k)}} \prod_j |m_j|^{\alpha_j} \exp \Big( -\frac{\rho}{2} |m_j| \Big).
\end{align*}
Combining this estimate with Lemma~\ref{lem:elementarycalculus} we obtain the statement.
\end{proof}

\begin{corollary}\label{cor:12FouriercontractUF}
Let $h(t,n), c(t,n)$ be as in Lemma~\ref{lem:12Fouriercontract}. Then, $h(t,n) = c(t,n)$ for all $n \in \mathbb{Z}^\nu$ and all $0 < t \le \min(t_0, \rho^\nu/(C_1 B 4^\nu C_0^\nu |\omega|))$, where $C_1$ is an absolute constant.
\end{corollary}

\begin{proof}
Combining Corollary~\ref{cor:12Fouriercontract} with the estimate \eqref{summationbasic} we get
\begin{align*}
|h(t,n) & - c(t,n)| \le \frac{B^{k+1} (4^{\nu+1} C_0^\nu \rho^{-\nu} |\omega| t)^{k}}{k!} \sum_{\alpha \in \mathbb{B}^{(k)}} \prod_j \alpha_j! \\
& \le \frac{B^{k+1} (4^{\nu+1} C_0^\nu \rho^{-\nu} |\omega| t)^{k}}{k!}  (2N)^k
\end{align*}
with $N=k+1$. Due to Stirling's formula,
\[
k! \gtrsim k^k e^{-k},\quad (k!)^{-1}(2N)^k\lesssim (2e)^{k}.
\]
Take here $0 < t < \rho^\nu/(C_1 B 4^\nu C_0^\nu |\omega|)$ with $C_1\gg 1$ being an absolute constant. Then
\begin{align*}
\lim_{k\to \infty} \frac{B^{k+1} (4^{\nu+1}C_0^\nu \rho^{-\nu} |\omega| t)^{k}}{k!}  (2N)^k=0.
\end{align*}
This implies $h(t,n) = c(t,n)$ for all $n \in \mathbb{Z}^\nu$ for any $0 < t \le \min(t_0, \rho^\nu/(C_1 B_0 4^\nu C_0^\nu |\omega|))$.
\end{proof}

\begin{corollary}\label{cor:12FouriercontractU}
Let
\begin{equation}\label{eq:12FourierexpUPr}
u(t,x) = \sum_{n \in \mathbb{Z}^\nu} c(t,n) \exp(i x n \omega), \quad v(t,x) = \sum_{n \in \mathbb{Z}^\nu} h(t,n) \exp(i x n \omega)
\end{equation}
with $|c(t,n)|, |h(t,n)| \le B \exp(-\rho |n|)$, $n \in \mathbb{Z}^\nu$, $\rho > 0$. Assume that both $u,v$ obey the KdV equation,
$$
\partial_t u + \partial^3_x u + u \partial_x u = 0, \quad \quad \partial_t v + \partial^3_x v + v \partial_x v = 0
$$
for $0 \le t \le t_0$, $x \in \mathbb{R}$. Assume also $v(0,x) = u(0,x)$ for all $x \in \mathbb{R}$. Then, $v(t,x) = u(t,x)$ for $0 < t < \min(t_0,\rho^\nu/(C_1 B 4^\nu C_0^\nu |\omega|))$ and $x \in \mathbb{R}$. Here $C_1$ is an absolute constant.
\end{corollary}

\begin{proof}
Due to Lemma~\ref{lem:12Fourier2U}, the equations \eqref{eq:Fouriercontrac1} hold. Thus, $h(t,n), c(t,n)$ obey the conditions of Lemma~\ref{lem:12Fouriercontract}. Therefore the statement follows from Corollary~\ref{cor:12FouriercontractUF}.
\end{proof}

\begin{proof}[Proof of Theorem A]
It follows from Corollary~\ref{cor:cestimatesBkRE} that there exists an absolute constant $C_2$ such that for each $0 < t < \kappa^\nu/(B_0 C_2^\nu |\omega|)$ and $n \in \mathbb{Z}^\nu$, the following limit
$$
c^{(0)}(t,n) = \lim_{k \rightarrow \infty} c_{k}(t,n)
$$
exists and obeys
$$
|c^{(0)}(t,n) - c_{k-1}(t,n)| \le 2 B_0^{k+1} (4^{\nu+1} C_0^\nu \kappa^{-\nu} |\omega| t)^{k} \exp \Big( -\frac{\kappa}{4} |n| \Big).
$$
It follows from Corollary~\ref{cor:12system1cestimates} that
$$
|c^{(0)}(t,n)| \le 2 B_0 \exp \Big( -\frac{\kappa |n|}{2} \Big).
$$
Using these estimates, one derives from \eqref{eq:12FourierFdefi1a} the following system of equations for $c^{(0)}(t,n)$, $n \in \mathbb{Z}^\nu$:
$$
c^{(0)}(t,n) = c(n)\exp(i t(n\omega)^3) - \frac{i n \omega}{2} \int_0^t \exp(i(t - \tau)(n\omega)^3) \sum_{m_1, m_2 \in \mathbb{Z}^\nu : m_1 + m_2 = n} c^{(0)}(\tau,m_1) c^{(0)}(\tau, m_2) \, d \tau.
$$
Due to Lemma~\ref{lem:12Fourier2}, the function
$$
u = \sum_{n \in \mathbb{Z}^\nu} c^{(0)}(t,n) \exp(i x n \omega)
$$
obeys the following differential equation,
$$
\partial_t u = -\partial^3_x u + v,
$$
with $u(0,x) = u_0(x)$, where
\begin{align*}
u_0(x) & = \sum_{n \in \mathbb{Z}^\nu} c(n) \exp(i x n \omega), \\
v & = -\sum_{n \in \mathbb{Z}^\nu} \sum_{m_1, m_2 \in \mathbb{Z}^\nu : m_1 + m_2 = n} \frac{i (m_1 + m_2) \omega}{2} c^{(0)}(t, m_1) c^{(0)}(t, m_2) \exp(i x n \omega).
\end{align*}
Clearly, $v(t,x) = -\frac{1}{2} \partial_x (u^2(t,x))$. This proves the existence statement in Theorem~A.

The uniqueness statement in Theorem~A is due to Corollary~\ref{cor:12FouriercontractU}. This finishes the proof of the theorem.
\end{proof}

\begin{remark}\label{rem:continuity}
It follows from Corollary~\ref{cor:12system1cestimates} that the derivatives $\partial_t u,\partial^\alpha_x u$ are continuous for $0 \le t \le t_0$, $x \in \mathbb{R}$, $\alpha \le 3$.
\end{remark}

Now we turn to the proof of Theorem~B.

\begin{proof}[Proof of Theorem~B]
The uniqueness statement in Theorem~B follows from the uniqueness statement in Theorem~A by standard arguments.

Recall the following fundamental result by Lax; see \cite{Lax}. Let $u(t,x)$ be a function defined for $0 \le t < t_0$, $x \in \mathbb{R}$ such that $\partial^\alpha_x u$ exists and is continuous and bounded in both variables for $0 \le \alpha \le 3$. Assume that $u$ obeys the KdV
equation
$$
\partial_t u + \partial_x^3 u + u \partial_x u = 0.
$$
Consider the Schr\"odinger operators
\begin{equation} \label{eq:1-1u}
[H_t \psi](x) = -\psi''(x) +  \frac16 u(t,x) \psi(x), \qquad x \in \IR.
\end{equation}
Then,
\begin{equation}\label{eq:isospectral}
\sigma(H_t) = \sigma(H_0) \quad \text{ for all } t.
\end{equation}

We now invoke the following statements from \cite{DG}, see Theorems~A and B in that work. Consider the Schr\"odinger operator
$$
[H \psi](x) = -\psi''(x) +  V(x) \psi(x), \qquad x \in \IR,
$$
where $V(x)$ is a real quasi-periodic function
$$
V(x) = \sum_{n \in \zv}\, c(n) e^{i x n \omega}.
$$
Assume that the Fourier coefficients $c(m)$ obey
$$
|c(m)| \le \ve \exp(-\kappa_0 |m|).
$$
Assume that the vector $\omega$ satisfies the following Diophantine condition:
\begin{equation}\label{eq:1-8}
|n \omega| \ge 2 \pi a_0 |n|^{-b_0}, \quad n \in \zv \setminus \{0\}
\end{equation}
with some $0 < a_0 < 1$, $\nu-1 < b_0 < \infty$.

Then, there exists $\ve_0 = \ve_0(\ka_0, a_0, b_0) > 0$ such that if $\ve \le \ve_0$, the spectrum of $H$ has the following description \cite[Theorem~A]{DG},
$$
\sigma(H) = [E_{\min} , \infty) \setminus \bigcup_{m \in \zv \setminus \{0\} :} (E^-_m, E^+_m),
$$
where the gaps $(E^-_m, E^+_m)$ obey $E^+_m - E^-_m \le 2 \ve \exp(-\frac{\kappa_0}{2} |m|)$ \cite[Theorem~B]{DG}. Furthermore, there exists $\ve^\zero = \ve^\zero(\ka_0, a_0, b_0) > 0$ such that if the gaps $(E^-_m, E^+_m)$ obey $E^+_m - E^-_m \le \ve \exp(-\kappa |m|)$ with $\ve < \ve^\zero$, $\kappa > 4 \kappa_0$ then, in fact, the Fourier coefficients $c(m)$ obey $|c(m)| \le (2\ve)^{1/2} \exp(-\frac{\kappa}{2} |m|)$ \cite[Theorem~B]{DG}.

Let
$$
u_0(x) = \sum_{n \in \zv}\, c_0(n) e^{i x n \omega}.
$$
Assume that the vector $\omega$ satisfies the Diophantine condition \eqref{eq:1-8}. Set
$$
\ve^\one = \min \left( \frac{\ve_0(a_0,b_0,\kappa_0/8)^4}{2} , \frac{\ve^\zero(a_0,b_0,\kappa_0/2)}{4} \right).
$$
Assume that the Fourier coefficients $c_0(m)$ obey
\begin{equation}\label{eq:1-fouriercoeffa}
|c_0(m)| \le \ve^{\one} \exp(-\kappa_0 |m|).
\end{equation}
By Theorem A there exists $t_0 = t_0(a_0,b_0,\kappa_0) > 0$ such that for $0 \le t < t_0$, $x \in \mathbb{R}$, one can define a function
$$
u(t,x) = \sum_{n \in \zv}\, c(t,n) e^{i x n \omega},
$$
with $|c(t,n)| \le 2 \ve^\one \exp(-\frac{\kappa_0}{2} |n|)$, which obeys equation \eqref{eq:1.KdV} with the initial condition $u(0,x) = u_0(x)$. Moreover, due to Remark~\ref{rem.cont} and Corollary~\ref{cor:12system1cestimates}, the derivatives $\partial_t u,\partial^\alpha_x u$ are continuous for $0 \le t \le t_0$, $x \in \mathbb{R}$, $\alpha \le 3$.

Now assume that for some $T > 0$, one can define a function
$$
u(t,x) = \sum_{n \in \zv}\, c(t,n) e^{i x n \omega},\quad 0 \le t \le T
$$
with $|c(t,n)| \le (\ve^\one)^{1/4} \exp(-\frac{\kappa_0}{8} |n|)$ for $0 \le t \le T$ and $n \in \zv$, which obeys equation \eqref{eq:1.KdV} for $0 \le t \le T$ with the initial condition $u(0,x) = u_0(x)$. Moreover, assume that the derivatives $\partial_t u,\partial^\alpha_x u$ are continuous for $0 \le t \le T$, $x \in \mathbb{R}$, $\alpha \le 3$. As we mentioned, due to Theorem~A, such a $T$ exists.

We claim that, in fact, the Fourier coefficients obey $|c(t,n)| \le 6(2 \ve^\one)^{1/2} \exp(-\frac{\kappa_0}{4} |n|)$. Indeed, let $H_t$ be as in \eqref{eq:1-1u}. Since
\begin{equation}\label{eq:1-fouriercoeffa6}
\frac{|c_0(m)|}{6}<|c_0(m)| \le \ve^{\one} \exp(-\kappa_0 |m|).
\end{equation}
and $(\ve^\one)^{1/4} < \ve_0(a_0,b_0,\kappa_0/8)$, \cite[Theorem~A]{DG} applies with $\frac{u_0}{6}$ in the role of $V$. Let $(E^-_m(t), E^+_m(t))$ be the gaps in the spectrum of $H_t$. By \cite[Theorem~B]{DG}, $E^+_m(0) - E^-_m(0) \le 2 \ve^\one \exp(-\frac{\kappa_0}{2} |m|)$. Thus, by \eqref{eq:isospectral}, $E^+_m(t) - E^-_m(t) \le 2 \ve^\one \exp(-\frac{\kappa_0}{2} |m|)$. Since $2 \ve^\one < \ve^\zero(a_0,b_0,\kappa_0/2)$, by \cite[Theorem~B]{DG}, applied to $\frac{u(t,x)}{6}$ in the role of $V$ the Fourier coefficients $\frac{c(t,m)}{6}$ obey $\frac{|c(t,m)|}{6} \le (2 \ve^\one)^{1/2} \exp(-\frac{\kappa_0}{4} |m|)$, as claimed. Due to Theorem~A applied to $u(T,x)$ in the role of the initial condition there exists $t_0 = t_0(a_0,b_0,\kappa_0/4) > 0$ such that for $0 \le t \le t_0$, $x \in \mathbb{R}$, one can define a function
$$
\tilde u(t,x) = \sum_{n \in \zv}\, \tilde c(t,n) e^{i x n \omega},
$$
with $|\tilde c(t,n)| \le 12 (2\ve^\one)^{1/2} \exp(-\frac{\kappa_0}{8} |n|) < (\ve^\one)^{1/4} \exp(-\frac{\kappa_0}{8} |n|)$, which obeys equation \eqref{eq:1.KdV} with the initial condition $\tilde u(0,x) = u(T,x)$. Moreover, due to Remark~\ref{rem.cont} and Corollary~\ref{cor:12system1cestimates}, the derivatives $\partial_t \tilde u, \partial^\alpha_x \tilde u$ are continuous for $0 \le t \le t_0$, $x \in \mathbb{R}$, $\alpha \le 3$. Consider the extension of the function $u(t,x)$ for $0  \le t \le T_1$, $x \in \mathbb{R}$ by setting $u(t,x) := \tilde u(t-T,x)$, $T \le t \le T_1$  with $T_1 = T + t_0$. Then,
$$
u(t,x) = \sum_{n \in \zv}\, c(t,n) e^{i x n \omega},\quad 0 \le t \le T_1
$$
with $|c(t,n)| \le (\ve^\one)^{1/4} \exp(-\frac{\kappa_0}{8} |n|)$ for $0 \le t \le T_1$ and $n \in \zv$ and it obeys equation \eqref{eq:1.KdV} for $0 \le t \le T_1$ with the initial condition $u(0,x) = u_0(x)$.

This argument shows that there exists
$$
u(t,x) =\sum_{n \in \zv}\, c(t,n) e^{i x n \omega},
$$
with $|c(t,n)| \le (\ve^\one)^{1/4} \exp(-\frac{\kappa_0}{8} |n|)$, which obeys the KdV equation \eqref{eq:1.KdV} for all $t$ and also the initial condition $u(0,x) = u_0(x)$. This finishes the proof of the theorem.
\end{proof}

\end{document}